\documentclass[11pt,reqno]{amsart}

\title[Prandtl and hydrostatic Euler equations]{On the local well-posedness of the Prandtl and the hydrostatic Euler equations with multiple monotonicity regions}
\date{\today}
\author{Igor Kukavica}
\address{University of Southern California, Los Angeles, CA 90089}
\email{kukavica@usc.edu}

\author{Nader Masmoudi}
\address{Courant Institute, New York University, New-York, NY 10012.}
\email{masmoudi@cims.nyu.edu}

\author{Vlad Vicol}
\address{Princeton University, Princeton, NJ 08544}
\email{vvicol@math.princeton.edu}

\author{Tak Kwong Wong}
\address{University of Pennsylvania, Philadelphia, PA 19104}
\email{takwong@math.upenn.edu}

\usepackage[margin=1in]{geometry}

\usepackage{amsfonts,amsmath,latexsym,amssymb,verbatim,amsbsy,amsthm,mathrsfs}

\usepackage{times}

\usepackage{hyperref}

\usepackage[usenames,dvipsnames]{xcolor}

\def\indeq{\qquad}

\theoremstyle{plain}
\newtheorem{theorem}{Theorem}[section]

\newtheorem{lemma}[theorem]{Lemma}

\theoremstyle{definition}
\newtheorem{remark}[theorem]{Remark}

\def\tilde{\widetilde}
\numberwithin{equation}{section}

\def\RR{\mathbb R}
\def\ZZ{\mathbb Z}
\def\HH{\mathbb H}
\def\DD{\mathcal D}
\def\DDA{\DD_{\rm a}}
\def\DDM{\DD_{\rm m}}
\def\AA{\mathcal A}
\def\eps{\varepsilon}

\def\dyinv{\partial_y^{-1}}

\def\bu{\bar u}
\def\by{\bar y}
\def\dbyinv{\partial_{\by}^{-1}}

\def\tom{\tilde \omega}
\def\tu{\tilde u}

\def\dx{\partial_x}

\def\dy{\partial_y}
\def\dyy{\partial_{yy}}
\def\dt{\partial_{t}}

\begin{document}

\baselineskip=15pt

\begin{abstract}
We find a new class of data for which the Prandtl boundary layer equations and the hydrostatic Euler equations are locally in time well-posed. In the case of the Prandtl equations, we assume that the initial datum
$u_0$ is monotone on a number of intervals (on some strictly increasing on some strictly decreasing) and analytic on the complement and show that the local existence and uniqueness hold. The same is true for the hydrostatic Euler equations except that we assume this for the vorticity $\omega_0=\partial_y u_0$.
\end{abstract}



\maketitle

\section{Introduction}\label{sec:intro}

In this paper we address the local in time well-posedness for two systems of partial differential equations which arise when considering singular limits in fluid mechanics: the Prandtl boundary layer equations and the hydrostatic Euler equations.
The common features of these systems are:
\begin{itemize}
 \item they are {formally} derived using singular asymptotic expansions,
 \item well-posedness theory established so far is rigid:
the existence works were obtained in the class of real-analytic functions or under certain monotonicity/convexity assumptions, and
 \item fundamental instabilities make the systems ill-posed in Sobolev spaces.
\end{itemize}

\subsection{The Prandtl Equations} 
\label{sec:Prandtl}
When considering the inviscid limit of the 2D~Navier-Stokes equations for the velocity field $(u^{NS},v^{NS})$ on a bounded domain with Dirichlet boundary conditions, one is faced with the fundamental difficulty of mismatch between the boundary conditions of the viscous flow ($u^{NS} (x,y,t) = v^{NS}(x,y,y) = 0$ on the boundary) and the Euler flow ($v^E(x,y,t)=0$ on the boundary). We refer the reader for example to~\cite{ConstantinWu95,ConstantinWu96}
for vanishing viscosity results in domains without boundary,
\cite{Kato84b,TemamWang97b,Masmoudi98,Wang01,Kelliher07,Masmoudi07,LopesMazzucatoLopes08,
LopesMazzucatoLopesTaylor08,MazzucatoTaylor08,Maekawa12}
for vanishing viscosity results with the Dirichlet boundary conditions,
and \cite{BeiraoCrispo10,BeiraoCrispo12,MasmoudiRousset12}
for the Navier type conditions.

To overcome this difficulty, Prandtl~\cite{Prandtl1904}
introduced an ansatz 
$(u^{NS},v^{NS})(x,y,t) \approx (u^{E},v^{E})(x,y,t)$
valid away from a thin neighborhood of the boundary, called the
boundary layer.
Inside the boundary layer, in order to account for the large gradients in the normal direction, the Navier-Stokes flow should asymptotically behave as $(u^{NS},v^{NS})(x,y,t) \approx (u,\sqrt{\nu} v)(x,y/\sqrt{\nu},t)$, where $\nu$ is the kinematic viscosity.
Retaining the leading order terms in this matched asymptotic expansion, one formally arrives at the Prandtl equations (see, e.g.~\cite{Oleinik66,OligerSundstrom78,EEngquist97,OleinikSamokhin99,SammartinoCaflisch98a,CaflischSammartino00,E00,HongHunter03,GuoNguyen10,Grenier00,XinZhang04,GarganoSammartinoSciacca09,GerardVaretDormy10,GerardVaretNguyen12,MasmoudiWong12a,MasmoudiWong12b,AlexandreWangXuYang12} for results on the Prandtl equations).

The two dimensional Prandtl equations for the unknown (tangential) velocity $u= u(x,y,t)$ in the boundary layer read
\begin{align}
  & \partial_t u - \dyy u + u \dx u + v \dy u = - \dx P \label{eq:P:1}\\
  & \dx u + \dy v = 0\label{eq:P:2}\\
  & \dy P = 0 \label{eq:P:3}
\end{align} in $\HH = \{(x,y)  \in {\mathbb R}^2 : y>0\}$, where $y$ is the normal variable in the boundary layer. The equations \eqref{eq:P:1}--\eqref{eq:P:3} are augmented with the no-slip and the no-influx boundary conditions
\begin{align}
&  u(x,y,t)|_{y=0} = v(x,y,t)|_{y=0}  = 0 \label{eq:P:4}
\end{align} for $t>0$, and the matching conditions with the Euler flow as $y\rightarrow \infty$, via the Bernoulli law
\begin{align}
  &\lim_{y\rightarrow \infty} u(x,y,t) = U(x,t)\label{eq:P:5}\\
  &\dx P(x,t) = - (\partial_t + U(x,t)  \dx) U(x,t)\label{eq:P:6}
\end{align}
for $x\in {\mathbb R}$, $t>0$, where $U(x,t)$ is given by the trace at $y=0$ of the tangential component of the underlying Euler flow. Note that the vertical component of the velocity $v=v(x,y,t)$ is determined from $u$ (it is a {\em diagnostic variable}) via \eqref{eq:P:2} and \eqref{eq:P:4}:
\begin{align}
v(x,y,t) = - \dyinv \dx u(x,y,t)
\label{eq:dyinv:def}
\end{align}
for all $(x,y,t) \in \HH \times [0,\infty)$. 
Here and throughout the text, for any function $w \colon \HH \to \RR$, we denote
\begin{align}
\dyinv w(x,y) = \int_0^y w(x,z) dz.
\label{EQ07}
\end{align}
The pressure gradient $\partial_x P$ appearing on the right side of \eqref{eq:P:1} is given by \eqref{eq:P:3} and \eqref{eq:P:6} as the trace at $y=0$ of the underlying Euler pressure gradient.

From the mathematical point of view, the formal derivation of the Prandtl equations raises two intimately connected fundamental questions: 
\begin{itemize}
\item In what sense are the Prandtl equations well-posed (at least locally in time)?
\item In which space can we rigorously justify the Prandtl asymptotics?
\end{itemize}
In this paper we address the first one. The well-posedness of the Prandtl equations has been established so far only in three particular settings: either for initial data that is monotone in the $y$ variable, or for data that is real-analytic in the $x$ variable, or for initial data that changes monotonicity (in a non-degenerate way) in the $y$ variable and is of Gevrey-class in the $x$ variable. The results in this paper give a fourth regime in which the Prandtl boundary layer equations are well-posed (cf.~Theorems~\ref{thm:Prandtl:particular} and~\ref{thm:Prandtl:general} below).

\subsection*{Summary of previous results}
In the two-dimensional case, if one assumes the initial data $u$ is monotonic (and the matching Euler flow $U$ has the correct sign), using the Crocco transform it was shown in~\cite{Oleinik66}  that the equations have a unique local in time solution (cf.~also \cite{OleinikSamokhin99}). A recent, energy-based proof of this result was obtained in~\cite{MasmoudiWong12a} by appealing to a special nonlinear cancellation in the equations (cf.~also \cite{Grenier00b,AlexandreWangXuYang12}). In this {\em monotone setting}, if the pressure gradient is favorable $\partial_x P \leq 0$ the local solution can be extended globally in time~\cite{XinZhang04}. A special finite time blowup solution of the Prandtl equations was constructed in~\cite{EEngquist97} when $U=P=0$, although the corresponding inviscid problem is well-posed in a weak sense~\cite{HongHunter03}.

The second setting where the Prandtl equations are locally well-posed is the {\em analytic setting}. Using an Abstract Cauchy-Kowalewski theorem~\cite{Asano88}, it was shown in \cite{SammartinoCaflisch98a,SammartinoCaflisch98b} that if the initial data for the Prandtl equations are real analytic with respect to $x$ and $y$ (and so is the underlying Euler flow $U$), then the equations are locally well-posed in this class. The requirement of analyticity in the normal variable was removed in~\cite{CannoneLombardoSammartino01} (cf.~also \cite{CaflischSammartino00, GarganoSammartinoSciacca09}).  Recently, in~\cite{KukavicaVicol13a}, an energy-based proof of the local well-posedness result was given, assuming the initial data is real-analytic with respect to $x$ only. In addition, the result allows $u-U$ to decay at an algebraic rate as $y \to \infty$. In the three-dimensional case, the Prandtl equations are only known to be well-posed in this analytic setting.

Removing the condition of real-analyticity in the normal variable $y$ is relevant in the context of the finite time blowup for the Prandtl equations considered in \cite{EEngquist97}. The initial datum $u_0(x,y)$ shown there to yield finite time singularities has compact support in the $y$ variables and hence cannot be real-analytic. An example is given by the initial data $u_0(x,y) = - x \exp(-x^2) a_0(y)$, where $a_0(y) = f(Ry)$, $f$ is a positive, compact support bump function, and $ 4 R < \|f\|_{L^3}^{3/2} \|f'\|_{L^2}^{-1}$. Therefore, for this initial datum we have both local existence (cf.~\cite{CannoneLombardoSammartino01,KukavicaVicol13a})  and finite time blowup (cf.~\cite{EEngquist97}) of \eqref{eq:P:1}--\eqref{eq:P:6} with $U=P=0$.

At this stage we point out that in the Sobolev (and even $C^\infty$) category the equations have recently been shown to be ill-posed in the sense of Hadamard~\cite{GerardVaretDormy10} due to high-frequency instabilities in the equations linearized about certain non-monotonic shear flows (cf.~also \cite{CowleyHockingTutty85,GuoNguyen10,GerardVaretNguyen12}). The instabilities exhibited in \cite{GerardVaretDormy10,GerardVaretNguyen12} do not however preclude the well-posedness of the system in Gevrey spaces with index between $2$ and $1$ (the Gevrey-class $1$ is the class of real-analytic functions). In this direction a very recent result~\cite{GerardVaretMasmoudi13} shows that the equations are locally well-posed if the initial data lies in the Gevrey class $7/4$ and changes monotonicity in a non-degenerate way across the graph of a function of $x$. This result exhibits a nice interplay between the Gevrey class/analytic setting and the nonlinear cancellation available in the monotone setting.

\subsection*{Main results}
In this paper we address the following question for the two-dimensional Prandtl system.
Assume that on one part of the domain the initial data is given by a profile that is monotone increasing with respect to the normal variable and on another part of the domain the initial data is monotone decreasing (for example consider initial vorticity $\omega_0(x,y) = \sin(x) (1+y^2)^{-1}$, defined on the domain $[-\pi/2,\pi/2]\times \RR_+$).
{\em What is a sufficient condition to impose on the complement of these two regions, in order to ensure that the equations are locally well-posed?} We prove that such a sufficient condition is given by assuming the data is uniformly real-analytic with respect to the tangential variable in this complementary region (cf.~Theorems~\ref{thm:Prandtl:particular} and~\ref{thm:Prandtl:general} below).

The difficulties in establishing this result, as well as the analogous one for the hydrostatic Euler equations, are as follows. As opposed to~\cite{GerardVaretMasmoudi13}, we need to localize in the $x$ variable (instead of $y$) and real-analytic functions of $x$ cannot have compact support. Additionally, the localization in the $x$ variable suggests that one needs to solve for the monotone region in the presence of side boundary conditions, a construction that is not amenable using the existing energy-based tools. 

Moreover, at the technical level the norms used in the monotone region
and the analytic region are not compatible.  Indeed, in order to close the estimates in the monotone region at the level of Sobolev spaces, one works with the vorticity formulation of the Prandtl system with weights that match the number  $y$-derivatives of the solution (cf.~\cite{MasmoudiWong12a}). But it appears that in the analytic region these weights require the solution to be analytic in both $x$ and $y$, which is undesirable. Lastly, glueing the real-analytic and the monotone solutions is an issue since the latter lies merely in a finite order Sobolev space.

The main ideas that allow us to overcome these difficulties are the following. First, we observe that using a suitable change of normal variables, that depends on the underlying Euler flow $U$ (cf.~\cite{KukavicaVicol13a}),
we may construct the real-analytic solution (analytic in $x$ only, with algebraic matching with $U$ at the top of the layer) in a strip $I_a \times \RR_+$ (here $I_a$ is the {interval of analyticity}), without the use of lateral boundary conditions on $\partial I_a \times \RR_+$. In particular, this ``decouples'' the analytic estimates from the Sobolev $H^s$ estimates. This is possible because we can construct the analytic solution without integrating by parts in the $x$ variable (cf.~\cite{KukavicaVicol13a}). On the other hand, constructing the monotone solution at the level of Sobolev spaces essentially uses integration by parts in $x$. We overcome this problem by extending the initial data from $I_m \times \RR_+$ (here $I_m$ is the {interval of monotonicity}) to $\RR \times \RR_+$ while preserving the monotonicity, and use the results in \cite{MasmoudiWong12a} to construct the monotone solutions on the entire half-plane. Note that this extension cannot decay as $x\to \infty$ because then the uniform monotonicity condition would be violated. In order to work with these non-decaying solutions we appeal to a {\em locally uniform spaces} (cf.~\eqref{eq:Hs:uloc}--\eqref{eq:Linfty:uloc} below). This setting also allows us to extend the results in \cite{MasmoudiWong12a} to the case when $u$ is not $x$-periodic (cf.~Theorem~\ref{thm:Oleinik:global} below).

In the last and the crucial step, we glue the monotone solution defined on $\RR \times \RR_+$ with the analytic one defined in $I_a \times \RR_+$. The idea here is to use the {\em finite speed of propagation with respect to the $x$ variable}. Indeed, initially the monotone and analytic solutions agree on $(I_m \cap I_a) \times \RR_+$, which is non-empty. Thus, using that by continuity they are in fact both monotone there and that the real-analytic solution obeys $\sup_t \|u\|_{L^\infty(I_a \times \RR_+)} \leq M <\infty$, we prove that the two solutions agree at later times on a strip that shrinks with speed $2M$ with respect to $x$. That is, if $I_a \cap I_m \supset [a,b]$, then the solutions agree on $[a+Mt,b-Mt]\times \RR_+$ for all sufficiently small $t$. The uniqueness of the solution follows in a similar fashion. Our main results for the Prandtl system and their detailed proof are given in Section~\ref{sec:P:proof} below. 

\subsection{The Hydrostatic Euler Equations} \label{sec:hydro}

The two dimensional {hydrostatic Euler equations} for the unknown velocity field $(u,v) = (u,v)(x,y,t)$ and the scalar pressure $p=p(x,t)$, read
\begin{align} 
& \dt u + u \dx u + v \dy u + \dx p = 0 \label{eq:HE:1}\\
& \dx u + \dy v  = 0 \label{eq:HE:2}\\
& \dy p = 0 \label{eq:HE:3}
\end{align}
where $t\geq 0$ and the spatial domain is the infinite strip $\DD = \{ (x,y) \in \RR^{2} : 0 < y < 1\}$. The equations \eqref{eq:HE:1}--\eqref{eq:HE:3} are supplemented with the boundary condition
\begin{align} 
v = 0 \mbox{ on } \partial \DD \label{eq:HE:4}
\end{align}
where $\partial \DD = \{ (x,y) \in \RR^{2} : y=0 \mbox{ or } y=1\}$. 
The unknown variable $u$ is called {\em diagnostic}, while $v$ and $p$ are called {\em prognostic} since they may be computed from $u$. Indeed, from \eqref{eq:HE:2} and \eqref{eq:HE:4} we see that
\begin{align} 
v(x,y) = - \dyinv \dx u(x,y). \label{eq:HE:v:def}
\end{align}
The expression \eqref{eq:HE:v:def} and the boundary condition \eqref{eq:HE:4} naturally lead to the compatibility condition 
\begin{align*} 
\int_{0}^{1} u(x,y,t) dy = \psi(t)
\end{align*}
for $x \in \RR$ and $t \geq 0$, where $\psi$ is a function of time. 
Using a change of variables, 
we may without loss of generality consider consider this function $\psi$ to be zero, i.e.,
\begin{align} 
\int_{0}^{1} u_{0}(x,y) dy = \int_{0}^{1} u(x,y,t) dy= 0. \label{eq:mean}
\end{align}
On the other hand, integrating \eqref{eq:HE:1} in $y$ from $0$ to $1$ and using \eqref{eq:mean} we conclude that up to a function of time, which without loss of generality we set equal to zero, the pressure is determined by 
\begin{align} 
p(x,t) = - \int_{0}^{1} u^{2}(x,y,t) dy. \label{eq:p}
\end{align}

\subsection*{Summary of previous results} The hydrostatic Euler equations arise in two contexts in fluid dynamics: in modeling of the ocean and the atmosphere dynamics \cite{LionsTemamWang92a,LionsTemamWang92b,Pedlosky82,Masmoudi07b,TemamTribbia03,TemamZiane04}  and in the asymptotic limit of vanishing distance between two horizontal plates for the incompressible Euler equations~\cite{Lions96}.

Concerning the local existence and uniqueness of smooth solutions to the hydrostatic Euler equations \eqref{eq:HE:1}--\eqref{eq:HE:4} in the absence of lateral boundaries there are two main types of results. If one assumes a local Rayleigh condition, i.e., that $ \dyy u \geq \sigma > 0$ uniformly in $(x,t)$, then one can construct solutions in Sobolev spaces. This analysis was initiated in \cite{Brenier99,Brenier03,Grenier99}, which lead up to the recent energy-based approach in \cite{MasmoudiWong12b}. Note that if instead $u$ has inflection points, i.e., the Rayleigh condition is violated, the equations are Lipschitz ill-posed in Sobolev spaces \cite{Grenier00,Renardy09,Renardy11}. In the absence of a uniform Rayleigh condition the only available well-posedness results were obtained in \cite{KukavicaTemamVicolZiane11}, under the assumption that the initial data is real-analytic (cf.~also \cite{IgnatovaKukavicaZiane12} for the free surface case). The techniques in these works were inspired by the earlier results on analyticity for the Euler equations~\cite{LevermoreOliver97,OliverTiti01,KukavicaVicol11a,KukavicaVicol11b}.

In the presence of lateral boundaries the problem is even more challenging, since no local set of boundary conditions give rise to a well-posed problem~\cite{OligerSundstrom78,TemamTribbia03}. In this direction progress has been made regarding the linearized equations in~\cite{RousseauTemamTribbia05,RousseauTemamTribbia08}.

We note that in the absence of the concavity/convexity assumption, the finite time blowup of particular solutions to the hydrostatic Euler equations has been established very recently in~\cite{CaoIbrahimNakanishiTiti12} and \cite{Wong12}. The initial datum $u_0(x,y) = 1 + (1/3 - y^2) \sin(x)$ satisfies the blowup conditions in \cite{Wong12}, and is clearly real-analytic in both the $x$ and $y$ variables.  Therefore, for this initial datum we have both local existence (cf.~\cite{KukavicaTemamVicolZiane11}) and finite time blowup (cf.~\cite{CaoIbrahimNakanishiTiti12,Wong12}) of \eqref{eq:HE:1}--\eqref{eq:HE:4}.

\subsection*{Main results} In this paper we address the local existence of smooth solutions for the hydrostatic Euler equations for a larger class of initial data than previously considered in \cite{MasmoudiWong12b,KukavicaTemamVicolZiane11}. Similarly to the Prandtl case discussed in Section~\ref{sec:Prandtl} above, the question we study is the following.
Assume that the initial data satisfies $\dyy u_0 \geq \sigma > 0$ on a subset $\DD_{\rm m}^+$, and $\dyy u_0 \leq -\sigma$ on another subset $\DD_{\rm m}^-$ of $\DD$; what additional assumption on $u_0$ guarantees that one can construct a unique smooth solution, at least locally in time? We prove that if one assumes that $u_0$ is uniformly real analytic on the complement $\DDA$ of these concavity/convexity regions, then there exists $T>0$ and a solution $u$ on $[0,T]$, which lies in $H^s (\DDM)$ for some $s\geq 4$, and is real analytic on $\DDA$. That is, our initial data is allowed to be convex in certain regions of $\DD$ and concave in other regions, as long as in the transition region it is real analytic. The precise statements of our results for the hydrostatic Euler equations are given in Theorems~\ref{thm:hydro:main} and Remark~\ref{thm:hydro:general} below.

The main difficulties are similar as for the Prandtl system. Note that we cannot localize real analytic functions; thus we appeal to the interior analytic estimates, as in~\cite{KukavicaTemamVicolZiane11}, which decouple the analytic part from the $H^s$ part of the solution. In the proof of the local well-posedness in $H^s$ in \cite{MasmoudiWong12b} the key ingredient is a cancellation between the vorticity $\omega = \dy u$ and $v$ via integrating over $\DD$. Here an integration by parts in $x$ is essential, but under the current setting one cannot do this since the Rayleigh condition is not uniform over $\DD$. To overcome this difficulty, we extend the convex data to all of the strip, and use the idea in \cite{MasmoudiWong12b}, to construct a global (in $x$) convex solution. Lastly, we use the finite speed of propagation to glue the convex and analytic solutions and to prove uniqueness. Our main results for the hydrostatic Euler equations and their proofs are given in Section~\ref{sec:HE:proof} below.

\section{The main results for the Prandtl equations} \label{sec:P:proof}

\subsection{Local existence for Oleinik solutions in unbounded domains} \label{sec:P:Oleinik}

The main result of this subsection, Theorem~\ref{thm:Oleinik:global}, gives a construction of local in time solutions to the Prandtl equations under the Oleinik monotonicity condition,  using solely energy methods, even on domains that are unbounded with respect to the $x$-variable.  This result is a direct consequence of the construction of $x$-periodic $y$-monotone solutions given in \cite{MasmoudiWong12a}, and the finite speed of propagation in the $x$-variable inherent in the Prandtl equations.

We recall cf.~\cite{MasmoudiWong12a} the following function spaces.
For $s \geq 4$ an even integer, $\gamma \geq 1$, $\sigma > \gamma + 1/2$, $\delta\in(0,1)$, and a bounded interval $I \subset \RR$, we let
\begin{align}
H^{s,\gamma}(I) = \Bigl\{ \omega \colon I \times \RR_+ \to \RR:   \|\omega\|_{H^{s,\gamma}(I)} < \infty \Bigr\}
\end{align}
where 
\[
\|\omega\|_{H^{s,\gamma}(I)}^2 = \sum_{|\alpha| \leq s} \| (1+y)^{\gamma + \alpha_2} \dx^{\alpha_1} \dy^{\alpha_2} \omega\|_{L^2(I \times \RR_+)}^2.
\]
We also recall the space of $y$-monotone functions
\begin{align}
H^{s,\gamma}_{\sigma,\delta} (I)
&= \Bigl\{ \omega \in H^{s,\gamma}(I) : (1+ y)^\sigma \omega(x,y) \geq \delta \mbox{ for all } (x,y) \in I \times \RR_+, \notag\\
&\qquad \quad \sum_{| \alpha |\leq 2} | (1+y)^{\sigma+\alpha_2} \dx^{\alpha_1} \dy^{\alpha_2} \omega(x,y) | \leq \delta^{-2} \mbox{ for all } (x,y) \in I \times \RR_+ \Bigr\}.
\label{eq:weighted:Hs:def}
\end{align}
In order to simplify notations, we write
$H^{s,\gamma}(I)$ 
rather than
$H^{s,\gamma}(I \times \RR_+)$, 
i.e., the range of the $y$-variable is always assumed to be $\RR_+$.

For unbounded domains, e.g.~$I = \RR$, we need to deal with the fact that the condition 
\[
\omega(x,y) \geq \delta (1+y)^{-\sigma},
\] 
satisfied by functions in $H^{s,\gamma}_{\sigma,\delta}$, prevents the vorticity from decaying to zero when $|x| \to \infty$, and thus we cannot work in $L^2_x$-type spaces such as $H^{s,\gamma}(I)$. 

For this purpose we introduce the {\em uniformly local $H^{s,\gamma}$} space, as follows. Assume that there exists a finite length open interval $\bar I \subset \RR$, and a countable set of real numbers $\{a_j\}_{j \in{\mathbb N}}$ such that 
\[
I = \bigcup_{j \in{\mathbb N}} I_j, \qquad \mbox{where} \qquad I_j = a_j + \bar I.
\]
For instance if $I = (0,\infty)$ we may take $\bar I = (0,2)$ and $\{ a_j \}_{j\in{\mathbb N}} = \ZZ$. We then define
\begin{align}
H^{s,\gamma}_{\rm uloc}(I) 
= \Bigl\{ \omega \colon I \times \RR_+ \to \RR : \|\omega\|_{H^{s,\gamma}_{\rm uloc}(I)} = \sup_{j\in{\mathbb N}} \|\omega\|_{H^{s,\gamma}(I_j)} < \infty \Bigr\}
\label{eq:Hs:uloc}
\end{align}
and similarly
\begin{align}
H^{s,\gamma}_{\sigma,\delta,{\rm uloc}} (I)
&= \Bigl\{ \omega \in H^{s,\gamma}_{\rm uloc}(I) : (1+ y)^\sigma \omega(x,y) \geq \delta \mbox{ for all } (x,y) \in I \times \RR_+, \notag\\
&\qquad \quad \sum_{| \alpha |\leq 2} | (1+y)^{\sigma + \alpha_2} \dx^{\alpha_1} \dy^{\alpha_2} \omega(x,y) | \leq \delta^{-2} \mbox{ for all } (x,y) \in I \times \RR_+ \Bigr\}.
\label{eq:Linfty:uloc}
\end{align}
Note that the additional conditions for a function in $H^{s,\gamma}_{\rm uloc}$ to lie in $H^{s,\gamma}_{\sigma,\delta}$, are $L^\infty_{x,y}$-based conditions (rather than $L^2_{x,y}$-based) and thus we do not need to explicitly take an additional supremum over $j$. 

The main result of this subsection is as follows.

\begin{theorem}
\label{thm:Oleinik:global}
Let $I \subset \RR$ be an open interval, $s\geq 4$ be an even integer, $\gamma \geq 1$, $\sigma > \gamma+1/2$, and $\delta \in (0,1/4)$. Assume that the initial velocity obeys $u_0 - U_0 \in H^{s,\gamma-1}_{\rm uloc}(I)$, the initial vorticity satisfies $\omega_0 \in H^{s,\gamma}_{\sigma,2\delta,{\rm uloc}}(I)$, and that the outer Euler flow $U$ is sufficiently smooth (for example $\|U\|_{W^{s+9,\infty}_{t,x}} = \sup_t \sum_{0 \leq 2l \leq s+9} \| \partial_t^l U\|_{W^{s-2l+9,\infty}(I)} < \infty$ will suffice).
Then there exists $T>0$ and a 
smooth solution $u$ of \eqref{eq:P:1}--\eqref{eq:P:6} such that the velocity obeys $u-U \in L^\infty([0,T];H^{s,\gamma-1}_{\rm uloc} (I)) \cap C_w([0,T];H^s_{\rm uloc}(I))$ and the vorticity obeys
$\omega \in L^\infty([0,T];H^{s,\gamma}_{\sigma,\delta,{\rm uloc}}(I)) \cap C_w([0,T];H^s_{\rm uloc}(I))$.
When $I=\RR$ or $I={\mathbb T}$, the solution constructed above is the unique solution in this regularity class. When $ \partial I \neq \emptyset$, e.g.~if $I = (a,b)$, there exists a positive $M<\infty$ such that $u(t)$ is the unique solution with $u(t) - U(t) \in H^{s,\gamma-1}(I_t)$ and $\omega(t) \in H^{s,\gamma}_{\sigma,\delta,{\rm uloc}}(I_t)$, where 
$I_t=\{x\in{\mathbb R}:(x-M t,x+M t)\subseteq I\}$.
\end{theorem}

In the case $I=(a,b)$ we have $I_t=(a+Mt, b-Mt)$. When $s=4$ we need to also assume  that $\delta$ is chosen small enough so that $\| \omega_0 \|_{H_g^{s,\gamma}} \leq C \delta^{-1}$ (cf.~\cite[Equation (3.1)]{MasmoudiWong12a}). Here $C_w([0,T];X)$ stands for continuity with values in $X$, when the space $X$ is endowed with its weak topology.

For simplicity of the presentation we only give the proof in the case $I = \RR$, with $I_j = (j,j+2)$, and $j \in \ZZ$. The extension to the general case described in Theorem~\ref{thm:Oleinik:global} requires no additional arguments.

\begin{proof}[Proof of Theorem~\ref{thm:Oleinik:global}]
In order to appeal to the results in \cite{MasmoudiWong12a}, we need to work with solutions that are periodic in the $x$-variable. For this purpose fix an arbitrary $j \in \ZZ$. We extend the initial data $u_0, \omega_0$ and the outer Euler flow $U$ from the domain $I_j \times \RR_+ = (j,j+2)\times \RR_+$ to $x$-periodic functions $u_0^{(j)}, \omega_0^{(j)}, U^{(j)}$ defined on $\tilde I_j \times \RR_+ = (j-1,j+3) \times \RR_+$. The periodic extension is made in such a way that $u_0^{(j)}-U^{(j)} \in H^{s,\gamma-1}(\tilde I_j)$, with $\| u_0^{(j)}-U_0^{(j)} \|_{H^{s,\gamma-1}(\tilde I_j)} \leq C \| u_0 - U_0\|_{H^{s,\gamma-1}_{\rm uloc}(I)} = M_0$, and that $\omega_0^{(j)} \in H^{s,\gamma}_{\sigma,3\delta/2}(\tilde I_j)$, with $\| \omega_0^{(j)} \|_{H^{s,\gamma}_{\sigma,3\delta/2}(\tilde I_j)} \leq C \|\omega_0\|_{H^{s,\gamma}_{\sigma,2\delta,{\rm uloc}}} = M_1$. Since we have an $x$-periodic initial data that is monotone in $y$ and decays sufficiently fast as $y\to \infty$, we may apply the result in \cite[Theorem~2.2]{MasmoudiWong12a} {\em directly} in order to obtain a unique local in time solution with $u^{(j)} - U^{(j)} \in L^\infty([0,T_j];H^{s,\gamma-1}(\tilde I_j))$ and $\omega^{(j)} \in L^\infty([0,T_j]; H^{s,\gamma}_{\sigma,\delta}(\tilde I_j) )$, for some $T_j = T_j(M_0,M_1,s,\gamma,\sigma,\delta,\|U\|_{W^{s+9,\infty}_{t,x}})> 0$. At this stage we note that in view of the dependency of $T_j$ and the uniform local spaces considered, we may find $T>0$ such that $T\leq T_j$ for all $j\in \ZZ$, i.e., {\em the time of existence can be taken independent of $j$}.

We have so far constructed a countable family of solutions to the Prandtl equations, each of whom have a common life-span, but so that each solution $u^{(j)}$ is $\tilde I_j$-periodic in the $x$ variable. Using the transport structure of \eqref{eq:P:1} with respect to the $x$-variable we now proceed to ``glue'' the above constructed solutions. 

Fix some $j \in \ZZ$. At time $0$, by construction, the solutions $u^{(j)}$ and $u^{(j+1)}$ are identical on the interval $(j+1,j+2)$. 
By construction, we have that
\[
\sup_{j \in \ZZ} \sup_{t \in [0,T]} \left( \| u^{(j)} \|_{L^\infty(0,T;H^{s-1,\gamma}(\tilde I_j))} +\| \omega^{(j)} \|_{L^\infty(0,T;H^{s,\gamma}_{\sigma,\delta}(\tilde I_j))} \right) < \infty.
\]
Let
\[
M = \sup_{j \in \ZZ} \sup_{t \in [0,T]} \| u^{(j)}(t) \|_{L^\infty (\tilde I_j \times \RR_+)} < \infty
\]
be an upper bound on the speed of propagation in the $x$ variable. Note that since $s\geq 4$, $M$ may be bounded in terms of the $H^{s,\gamma-1}_{\rm uloc}$ energy $M_0$ and the $L^\infty_{t,x}$ norm of $U$.  It thus follows from the finite speed of propagation in the $x$ variable (transport structure) that for each 
\[
0 \leq t < T_* = \min\{ T, 2/M \}
\]
we have
\begin{align}
u^{(j)}(t,x,y) \equiv u^{(j+1)}(t,x,y), \qquad (x,y) \in (j+1 + Mt, j+2 - Mt) \times \RR_+.
\label{eq:monotone:glueing}
\end{align}
The proof of \eqref{eq:monotone:glueing} uses the ideas of the uniqueness proof in \cite[Section 6.2]{MasmoudiWong12a}. We  establish that the function 
\[
g = \omega^{(j+1)} - \omega^{(j)} - ( u^{(j+1)} - u^{(j)}) \partial_y \log( \omega^{(j)})
\]
obeys
\begin{align}
\| g(t) \|_{L^2( (j+1 + Mt, j+2 - Mt) \times \RR_+ )} \leq \| g(0) \|_{ (j+1,j+2)\times \RR_+} \exp (Ct)
\label{eq:monotone:glueing:estimate}
\end{align}
for $t \in [0,T_*)$ with some $C>0$. This implies that $g \equiv 0$ on $[0,T_*)$, and in view of the boundary conditions for the Prandtl equations we furthermore obtain \eqref{eq:monotone:glueing}. The detailed proof of the estimate \eqref{eq:monotone:glueing:estimate} is given in Section~\ref{sec:P:glueing} below (cf.~\eqref{eq:g:Gronwall}), for the case when one solution is analytic and the other one is monotone. The proof works however with no changes for the case of two monotone solutions and we omit the details to avoid redundancy.

In order to conclude the proof of the existence part of Theorem~\ref{thm:Oleinik:global}, we define a global (in $x$) solution $u(t,x,y)$ of the Prandtl equation on $[0,T_*) \times \RR \times \RR_+$ by
\begin{align}
u(t,x,y) = u^{(j)}(t,x,y) \qquad \mbox{whenever} \qquad x \in [j+1/2, j+3/2).
\end{align}
The uniform in $j$ bounds are inherited from the bounds obtained on the $u^{(j)}$'s individually. The proof of uniqueness follows by appealing again to the finite speed of propagation in the $x$ variable.
\end{proof}

\subsection{Local existence with multiple monotonicity regions} \label{sec:P:multiple:monotonicity}
The main result of this subsection, Theorem~\ref{thm:Prandtl:general} below, gives the local existence of the Prandtl equations in the presence of multiple monotonicity regions. 

For simplicity of the presentation, we only give the proof of a special case of our main result, Theorem~\ref{thm:Prandtl:particular} below, when the monotonicity regions cover the half plane minus the line $x=0$, around which the function is assumed to be analytic. Thus we treat the initial date $u_0$ which satisfy
  \begin{itemize}
   \item $\dy u_0 (x,y) < 0$ for $x < 0$ 
   \item $\dy u_0 (x,y) > 0$ for $x > 0$
   \item $u_0(x,y)$ is real analytic with respect to $x$ around $x=0$
  \end{itemize}
combined with suitable smoothness, decay, and compatibility conditions on the underlying tangential component of the Euler flow trace $U_0(x)$.

The precise setup is as follows.
Let $2 \tau_0 > 0 $ be the analyticity radius of $u_0(x,y)$ at $x=0$. The radius is assumed to be uniform with respect to $y$. Then the power series in $x$ for $u_0(x,y)$ converges with radius $\tau_0 >0$ at $(x,y) \in [- \tau_0, \tau_0] \times \RR_+$. Denote an interval of analyticity by 
\[
I_a = [-\tau_0,\tau_0].
\] 
 Quantitatively, we assume that $\bu_0$, as defined in \eqref{eq:bu:def} below, satisfies
\begin{align}
\| \bu_0 \|_{Y_{\tau_0}(I_a)} < \infty \label{eq:analytic:IC}
\end{align}
where the norm in $Y_{\tau_0}$ is defined in \eqref{eq:Y:tau:def} below. Here and throughout the paper we are implicitly using the fact that if a function is analytic on $(a,b)$ with radius $2\tau$, then it is analytic on $(a-\tau,b+\tau)$ with radius $\tau$.

In terms of the monotonicity assumptions on $u_0(x,y)$ for $x>0$ and for $x<0$, we assume that there exist $s\geq 4$, an even integer, $\gamma \geq 1, \sigma > \gamma + 1/2$, and $\delta \in (0,1/2)$ such that the initial velocity $u_0$ and the initial vorticity $\omega_0 = \dy u_0$ obey
  \begin{align}
   &u_0 - U_0 \in H^{s,\gamma -1}_{\rm uloc} (I_m^+), \quad u_0 - U_0 \in H^{s,\gamma - 1}_{\rm uloc}(I_m^-), \label{eq:monotone:IC:1}
  \end{align}
and
  \begin{align}
    &\omega_0 \in H^{s,\gamma}_{\sigma, 2 \delta,{\rm uloc}}(I_m^+), \quad -\omega_0 \in H^{s,\gamma}_{\sigma, 2 \delta,{\rm uloc}}(I_m^-) \label{eq:monotone:IC:2}
  \end{align}
with suitable assumptions on $U$ and on $\omega_0$ when $s=4$. Here we denote the {\em intervals of monotonicity} by
\begin{align*}
I_m^+ = [\tau_0/2,\infty) \quad \mbox{and} \quad I_m^- = (-\infty,-\tau_0/2].
\end{align*}
The spaces
$H^{s,\gamma -1}_{\rm uloc}(I)$ and
$H^{s,\gamma}_{\sigma, 4 \delta,{\rm uloc}}(I)$ are as defined in \eqref{eq:Hs:uloc} and \eqref{eq:Linfty:uloc} above.

Note that $I_a$ and $I_m^\pm$ have a non-empty overlap. With these notations, the statement whose proof we present in this section, is the following.

\begin{theorem}\label{thm:Prandtl:particular}
Let $s\geq 4$ be an even integer, $\tau_0>0$, $\gamma \geq 1$, $\sigma > \gamma+1/2$, and $\delta \in (0,1/4)$.
Assume that the initial velocity $u_0$ and vorticity $\omega_0$ obey \eqref{eq:analytic:IC}--\eqref{eq:monotone:IC:2} above, and that the underlying Euler flow $U$ is sufficiently smooth. If $s=4$, assume that $\delta$ is sufficiently small so that $\|\omega_0\|_{H_g^{s,\gamma}} \leq C \delta^{-1}$. Then there exists $T>0$ and a smooth solution $u$ of \eqref{eq:P:1}--\eqref{eq:P:6} on $[0,T]$ such that  $\bar u$, as defined in \eqref{eq:bu:def} below, belongs to $L^\infty([0,T];X_{{\eps_U \tau_0}}(I_a)) \cap L^2((0,T); Y_{\eps_U \tau_0}(I_a)) \cap L^2((0,T); Z_{\eps_U \tau_0}(I_a))$ on the analytic region, for some $\eps_U>0$, and that on the monotone regions we have 
$u-U \in L^\infty([0,T];H^{s,\gamma-1}_{\rm uloc}(I_m^{\pm})) \cap C_w([0,T];H^s_{\rm uloc}(I_m^{\pm}))$ 
and 
$\pm \omega \in L^\infty(H^{s,\gamma}_{\sigma,\delta,{\rm uloc}}(I_m^{\pm})) \cap C_w([0,T];H^s_{\rm uloc}(I_m^{\pm})) $. The solution $u$ is unique in this class.
\end{theorem}

The spaces $X_\tau, Y_\tau$, and $Z_\tau$ are defined in \eqref{eq:X:tau:def}--\eqref{eq:Z:tau:def} below. From the proof of Theorem~\ref{thm:Prandtl:particular}, given below, it is clear that the following more general statement holds.

\begin{theorem} \label{thm:Prandtl:general}
 Let $s\geq 4$ be an even integer, $\tau_0>0$, $\gamma \geq 1$, $\sigma > \gamma+1/2$, and $\delta \in (0,1/4)$. Assume that there exist open intervals $\{ I_{i}^{+} \}_{i=1}^{n_+}$ and $\{I_{k}^-\}_{k=1}^{n_-} \subset \RR$ such that the initial velocity $u_0(x,y)$ obeys 
$u_0 - U_0 \in H^{s,\gamma-1}_{\rm uloc}(I_i^\pm)$, 
the initial vorticity $\omega_0(x,y)$ satisfies 
$\pm \omega_0 \in H^{s,\gamma}_{\sigma,2\delta,{\rm uloc}}(I_i^\pm)$, 
and the underlying Euler flow $U(x,t)$ is such that $\| U \|_{W_{t,x}^{s+9,\infty}(I_i^\pm)}< \infty$. If $s=4$, assume that $\delta$ is chosen small enough depending on $\omega_0$. Assume also that there exist open intervals 
$\{J_j\}_{j=1}^{m} \subset \RR$ such that the initial vorticity $\omega_0$ is uniformly real analytic in $x$ with radius at least $2\tau_0$ there, i.e., $\omega_0 \in Y_{2\tau_0}(J_j)$, and that the underlying Euler flow $U(x,t)$ is uniformly (in $x$) real analytic on $J_j$ with radius at least $2\tau_0$ and norm bounded with respect to $t$. 
Lastly, assume the intervals of monotonicity $I_i^\pm$ and of analyticity $J_j$ cover the real line, that is $(\cup_{j} J_j) \cup (\cup_i I_i^+) \cup (\cup_k I_k^-)= \RR$. 
Then there exists $T>0$ and a smooth solution $u$ of \eqref{eq:P:1}--\eqref{eq:P:6} on $[0,T]$ which is monotone with respect to $y$ in $I_i^\pm \times \RR_+$, analytic with respect to $x$ in  $J_j \times \RR_+$, and is unique in this class.
\end{theorem}

We note that due of the uniform analyticity condition, the regions of overlap between the analyticity regions and the monotonicity regions have a uniform positive minimum length.
Also, it is not difficult to adjust the statement to accommodate an infinite number of intervals.

\begin{proof}[Proof of Theorem~\ref{thm:Prandtl:particular}]
The construction of solutions to \eqref{eq:P:1}--\eqref{eq:P:3} in the case of multiple monotonicity regions is carried out in several steps:
\begin{itemize}
 \item In subsection~\ref{sec:P:analytic}  we construct the real analytic solution defined on $I_a\times \RR_+$ and prove higher regularity properties with respect to the $y$ variable.
 \item In subsection~\ref{sec:P:monotone} we use the result of Section~\ref{sec:P:Oleinik} to construct the uniformly local monotone solution defined on $I_m^{\pm} \times \RR_+$.
 \item In subsection~\ref{sec:P:glueing} we use the finite speed of propagation in the $x$ variable to glue the monotone and real-analytic solutions. 
 \item In subsection~\ref{sec:P:uniqueness} the uniqueness of solutions in this class is established.
 \end{itemize}
The details for each of the above steps are given next.
\end{proof}

\subsubsection{Construction of the analytic solution}
\label{sec:P:analytic}

Since we are working with solutions of low regularity with respect to the $y$ variable, and since we need to allow for sub-exponential decay of the initial data and the solution as $y \to \infty$, we use as in \cite{KukavicaVicol13a} the Euler-dependent transformation
\begin{align}
\by &= A(x,t) y \label{eq:by:def}\\
\bu(x,\by,t) &= u(x,y,t) - (1-\phi(\by)) U(x,t)\label{eq:bu:def}
\end{align}
with
\begin{align}
\phi = \phi(\bar y) =\langle \by \rangle^{-\theta}, \quad \langle \by \rangle = (1+ {\by}^2)^{1/2}
\end{align}
where $\theta$ is sufficiently large, to be specified later,
and $A$ is the solution of
\begin{align}
\partial_t A + U \partial_x A = A \partial_x U, \qquad A|_{t=0} = 1.
\end{align}
It is also convenient to denote $a(x,t) = \log A(x,t)$.
This change of variables turns the Prandtl equation \eqref{eq:P:1}--\eqref{eq:P:6} into
\begin{align}
&\partial_t \bu - A^2 \partial_{\by \by} \bu + N(\bu) + L(\bu) = F \label{eq:PR}\\
&\bu(x,\by,t)|_{\by=0} = 0, \quad \lim_{\by \to \infty} \bu(x,\by,t) = 0 \label{eq:PR:BC}
\end{align}
with the corresponding initial condition,
where we denoted by
\begin{align}
N(\bu) = \bu \; \dx \bu - \dx \dbyinv \bu \; \partial_{\by}  \bu + \partial_x a\;  \dbyinv \bu \; \partial_{\by} \bu\label{eq:PR:N}
\end{align}
the nonlinear part,
by
\begin{align}
L(\bu) & = \dx \dbyinv \bu \; \partial_{\by} \phi \; U + \dx \bu (1-\phi) U + \partial_{\by} \bu  \left( \dbyinv\phi \; \dx U - \dx a \; \dbyinv \phi\; U\right) \notag\\
& \quad - \dbyinv \bu\; \dx a \; \partial_{\by} \phi\; U + \bu (1-\phi) \partial_x U\label{eq:PR:L}
\end{align}
the linear part, and by
\begin{align}
F = \left( \phi(1-\phi) + \dbyinv \phi \; \partial_{\by} \phi\right)U \partial_x U  - \dx a \; \partial_{\by} \phi \; \dbyinv \phi \; U^2  - A^2\; \partial_{\by \by} \phi \; U  -  \phi \; \dx P \label{eq:PR:F}
\end{align}
the force. 
From \eqref{EQ07}, we recall the notation $\dbyinv f(\cdot,\by) = \int_0^{\by} f(\cdot,z) dz$.

Note that by assumption  the function $U(x,t)$ is assumed to be uniformly real-analytic for $x \in [-2\tau_0, 2\tau_0]$ with the radius of analyticity bounded from below by some $2\tau_0>0$, on some time interval $[0,T]$. In particular $U$ is bounded (in $x$ and $t$), and by possibly reducing the time interval ($2 T \leq \tau_0/ \|U\|_{L^\infty_{x,t}}$ from a ballistic estimate) we can ensure that the values of $A(x,t)$ in $[-\tau_0,\tau_0]\times[0,T]$ depend only on those of $A_0 = 1$ and on $U(x,t), \dx U(x,t)$ in $[-2\tau_0,2\tau_0]\times [0,T]$.
Having established this, the existence and uniqueness (in the class of real analytic functions) of $A(x,t)$ on a short time interval $[0,T]$ follows from the classical Cauchy-Kowalewski theorem. By possibly reducing $T$ we may furthermore ensure that $1/2 \leq A(x,t)\leq 2$ on $I_a \times [0,T]$. Let the uniform radius of real-analyticity of the function $A(x,t)$ on $[-\tau_0,\tau_0] \times[0,T]$ be bounded from below by some $\tau_U>0$, and let its analytic norm on this set be bounded from above by some constant $G_U$.

In order to present the analytic a priori estimates, we introduce 
\begin{align}
\rho = \rho(\by) = \langle \by \rangle^\alpha
\end{align}
and for $\tau>0$, and an interval $I \subset \RR$ we define the real-analytic norms
\begin{align}
\|\bu\|_{X_\tau(I)}^2 &= \sum_{m \geq 0} \|\rho \dx^m \bu\|_{L^2(I \times \RR_+ )}^2 \tau^{2m} (m+1)^4 (m!)^{-2}
\label{eq:X:tau:def} \\
\|\bu\|_{Y_\tau(I)}^2 &= \sum_{m \geq 1} \|\rho \dx^m \bu\|_{L^2(I \times \RR_+ )}^2 \tau^{2m-1} m (m+1)^4 (m!)^{-2}
\label{eq:Y:tau:def} \\
\|\bu\|_{Z_\tau(I)}^2 &= \sum_{m \geq 0} \|A \rho \partial_{\by} \dx^m \bu\|_{L^2(I \times \RR_+)}^2 \tau^{2m} (m+1)^4 (m!)^{-2}.
\label{eq:Z:tau:def}
\end{align}
Then, similarly to estimate (3.26) in \cite{KukavicaVicol13a}, we have the {\em a priori estimate}
\begin{align}
& \frac{d}{dt} \Vert  \bu \Vert_{X_\tau(I_a)}^2  + \Vert \bu \Vert_{Z_\tau(I_a)}^2  \notag\\
& \qquad  \leq  C_a(1+\tau^{-2}) (1 + \Vert \bu \Vert_{X_\tau(I_a)}^2)^2    +  \left (\dot{\tau} + C_a + C_a \tau^{-1}  \Vert \bu \Vert_{Z_\tau(I_a)} \right) \Vert \bu \Vert_{Y_\tau(I_a)}^2\label{eq:ODE:analytic:1}
\end{align}
whenever $\tau = \tau(t) \leq \eps_U \tau_U$ for some $\eps_U \approx (1+G_U)^{-1} \in (0,1]$, and for some positive constant $C_a$ which depends on $\alpha$,  $\theta$, the analyticity radius $\tau_U$ and analytic norm $G_U$ of $A(x,t)$ on $[-\tau_0,\tau_0] \times [0,T]$.

While we do not provide full details for the above estimate, we wish to emphasize one important aspect. In \cite{KukavicaVicol13a} the estimates are for the half-space $\RR \times \RR_+$ so that there are no boundary terms in $x$. On the other hand, in the setting of this paper the estimates are considered in the strip $I_a \times \RR_+$, and hence integration by parts with respect to the $x$ variable is not permitted. Even so, at {\em no stage} in the derivation of \eqref{eq:ODE:analytic:1} was integration by parts with respect to $x$ used. The one derivative loss with respect to $x$ in the nonlinear term is compensated by requiring the analyticity radius to decrease fast enough, and so integration by parts in $x$ is not needed. 

There is just one more technical aspect which is different in obtaining \eqref{eq:ODE:analytic:1} for $I_a\times \RR_+$ instead of for $\RR \times \RR_+$. As opposed to the case of the whole line (with decay at infinity), when working on a finite interval the one-dimensional Agmon inequality for a function $f \colon I \to \RR$ has a lower order term, i.e., $\|f\|_{L^\infty} \leq C \|f\|_{L^2}^{1/2} \|f'\|_{L^2}^{1/2} + C_I \|f\|_{L^2}$. 
To obtain \eqref{eq:ODE:analytic:1} one repeatedly uses this estimate with $f = \dx^k \bu$, $k\geq 0$, and $I = I_a$. These lower order term do not create any difficulties in the estimates.

In order to conclude the analytic a priori estimates, let the analyticity radius $\tau(t)$ solve the differential equation
\begin{align}
  \frac{d}{dt} (\tau^{2}) + 4 C_a \tau(0) + 4 C_a \Vert \bu(t) \Vert_{Z_{\tau(t)}(I_a)}= 0, \quad \tau(0) = \eps_U \tau_0
  \label{eq:ODE:tau}
\end{align}
for some $\eps_U \in (0,1]$ as above.
In particular, due to continuity in time of both the solution $\bu(t)$ in $X_{\tau(t)}(I_a)$ and $Z_{\tau(t)}(I_a)$, and of $\tau(t)$, this means that on a short time interval, the second term on the right side of \eqref{eq:ODE:analytic:1} is negative, and therefore on this short time interval we have
\begin{align*}
\tau(t)^{2}\geq \tau(0)^2 - 4 C_a \tau(0)  t - 4 C_a t^{1/2} \left(\int_0^t \Vert \bu(s) \Vert_{Z_{\tau(s)}}^2\, ds\right)^{1/2} \geq \frac{\tau(0)^2}{4}
\end{align*}
and 
\begin{align*}
\frac{d}{dt} \Vert \bu \Vert_{X_\tau(I_a)}^2  +   \Vert \bu \Vert_{Z_\tau(I_a)}^2 \leq C_a(1+ 4\tau(0)^{-2}) (1 + \Vert \bu \Vert_{X_\tau(I_a)}^2)^2 
\end{align*}
which in particular implies that 
\begin{align*}
  \int_{0}^{t} \Vert \bu(s) \Vert_{Z_{\tau(s)}(I_a)}^2\; ds \leq  1 +  2 \Vert   \bu_0 \Vert_{X_{\tau_0}(I_a)}^2.
\end{align*}
Therefore, there exits $T_a >0$, depending solely on $\tau_0, \| \bu_0\|_{X_{\tau_0}}, U, \alpha, \theta$, such that
\begin{align}
\sup_{t\in[0,T_a]} \| \bu(t)\|_{X_{\tau(t)}(I_a)}^2 + \int_0^{T_a} \left( \| \bu (t) \|_{Z_{\tau(t)}}^2 + \| \bu (t) \|_{Y_{\tau(t)}}^2  \right) dt \leq 1 + 2 \Vert   \bu_0 \Vert_{X_{\tau(0)}(I_a)}^2 
\label{eq:analytic:final:estimate}
\end{align}
and 
\begin{align*}
2 \tau(t) \geq \tau(0) = \eps_U \tau_0
\end{align*}
on this time interval.

The above a priori estimates can be made rigorous by constructing the analytic solution via Picard iteration. 
Namely, let $\bu^{(0)} = \bu_0$ and 
\begin{align*}
\dt \bu^{(n+1)} - A^2 \partial_{\by \by} \bu^{(n+1)} = F - N(\bu^{(n)}) - L(\bu^{(n)}), \qquad \bu^{(n+1)}(0) = \bu_0
\end{align*}
for $n\geq 0$, with homogeneous boundary conditions at $\by=0$ and $\by = \infty$. As in \cite[Section~5]{KukavicaTemamVicolZiane11}, using \eqref{eq:analytic:final:estimate} the sequence $\bu^{(n)}$ may be shown to be contracting in the space
\[
L^\infty([0,T_a];X_{{\tau}}(I_a)) \cap L^2((0,T_a); Y_{\tau}(I_a)) \cap L^2((0,T_a); Z_{\tau}(I_a))
\]
from which the existence of the real-analytic solution follows. 
We note that in order to establish uniqueness of the analytic solution on $I_a \times \RR_+$, one still needs to estimate the difference of two solutions in an analytic norm (as in \cite[Section~6]{KukavicaTemamVicolZiane11}). This is due to the lack of lateral boundary conditions which prevents one from integrating by parts in $x$.

It is clear from the uniform real-analyticity of $A(x,t)$ on $I_a\times[0,T_a]$ and the substitutions \eqref{eq:by:def}--\eqref{eq:by:def}, that the real-analyticity of $\bu$ implies the real-analyticity of $u$, with comparable radii of analyticity.

At this stage we note that the $L^2$ in time control on the $Z_\tau$ norm of the analytic solution, combined with the parabolic character (in $t$ and $y$) of the equation obeyed by $\bu$, yields higher regularity properties of the analytic solution with respect to the $y$ variable. This fact will be needed later in the proof when we glue the real-analytic and the monotone solutions.

\begin{lemma}[\bf Vorticity of the analytic solution] 
\label{lem:vort:dy}
Let $\omega = \dy u$ be the vorticity associated with the real analytic solution $u$, computed from $\bu$ via \eqref{eq:bu:def}. Then we have
\begin{align}
\sup_{t\in[0,T_a]} \| \rho(y) \omega \|_{H^1_x L^2_y(I_a \times \RR_+)}^2 
+ \int_0^{T_a} \|\rho(y) \dy \omega(t) \|_{H^1_x L^2_y(I_a \times \RR_+)}^2 dt
\leq C_{\omega,a}^2
\label{eq:vort}
\end{align}
where $C_{\omega,a}$ depends on $\alpha, \theta, \tau_U, G_U,\tau(0)$, and $\|\bu_0\|_{X_{\tau(0)}}$, and $\rho(y) = \langle y \rangle^\alpha$ with $\alpha > 1/2$.
\end{lemma}
\begin{proof}
For simplicity of the presentation, we only give the proof of the estimates in $L^2_x L^2_y(I_a \times \RR_+)$ for $\rho \omega$ and $\rho \dy \omega$. The estimate with an additional $\dx$ derivative follows {\em mutatis mutandi}. 
Multiplying \eqref{eq:PR} by $-\partial_{\by \by} \bu\; \rho^2(\by)$ and integrating on $I_a \times \RR_+$ yields
\begin{align}
 &\frac{d}{dt} \| \rho \partial_{\by} \bu \|_{L^2(I_a\times \RR_+)}^2 +  \| \rho A \partial_{\by \by} \bu \|_{L^2(I_a\times \RR_+)}^2  \notag\\
 &\qquad \leq C \| \partial_{\by} \rho \partial_{\by} \bu \|_{L^2(I_a \times \RR_+)}^2 + C \| \rho \left( F - L(\bu) - N(\bu) \right) \|_{L^2(I_a\times \RR_+)}^2 
\label{eq:vorticity:analytic}
\end{align}
where we used the homogeneous boundary conditions with respect to $\by$ of $\bu$, the fact that $1/2 \leq A \leq 2$, and the Cauchy-Schwartz inequality.

By definition of the $Z_\tau$ norm in \eqref{eq:Z:tau:def}, we have
\begin{align}
\| \partial_{\by} \rho \partial_{\by} \bu \|_{L^2(I_a \times \RR_+)}^2 \leq C \|\bu\|_{Z_\tau}^2
\label{eq:vort:a:t}
\end{align}
for some constant $C$ that depends on $\alpha$.

We now recall the 1-dimensional Agmon inequalities with respect to the vertical variable
\begin{align}
\|f\|_{L^4_y(\RR_+)}  \leq C \| f\|_{L^2_y(\RR_+)}^{3/4} \|\dy f\|_{L^2_y(\RR_+)}^{1/4}
\label{eq:Agmon:y:1}
\end{align}
which holds if $f=0$ at $y=0$,
and a Hardy type inequality
\begin{align}
\| \dy^{-1} f\|_{L^\infty_y(\RR_+)} \leq C \| \rho f \|_{L^2_y(\RR_+)}
\label{eq:Agmon:y:2}
\end{align}
where $\rho(y) = \langle y \rangle^\alpha$ with $\alpha>1/2$. With respect to the horizontal variable we shall use
\begin{align}
\| f\|_{L^4_x(I)} \leq C \|f\|_{L^2_x(I)}^{3/4} \| \dx f \|_{L^2_x(I)}^{1/4} + C \|f\|_{L^2_x(I)} \leq C \|f\|_{H^1_x(I)}
\label{eq:Agmon:x:1}
\end{align}
and 
\begin{align}
\| f\|_{L^\infty_x(I)} \leq C \|f\|_{L^2_x(I)}^{1/2} \| \dx f \|_{L^2_x(I)}^{1/2} + C \|f\|_{L^2_x(I)} \leq C \|f\|_{H^1_x(I)}
\label{eq:Agmon:x:2}
\end{align}
for some positive constant $C$ that may depend on $I$.

Using the inequalities \eqref{eq:Agmon:y:1}--\eqref{eq:Agmon:x:2} and recalling the definition of $N(\bu)$ from \eqref{eq:PR:N}, and $\rho(\by) = \langle \by \rangle^\alpha \geq 1$, with $1/2 \leq A \leq 2$, we obtain
\begin{align}
\| \rho N(\bu) \|_{L^2} 
&\leq \| \rho \bu \|_{L^4} \| \dx \bu \|_{L^4} + \| \dx \dy^{-1} \bu\|_{L^\infty} \| \rho \dy \bu \|_{L^2} + C\| \dy^{-1} \bu\|_{L^\infty} \| \rho \dy \bu \|_{L^2} \notag\\
&\leq C \| \rho \bu \|_{H^2_x L^2_y}^{3/2} \| A \rho \partial_{\by} \bu\|_{H^2_x L^2_y}^{1/2} + C \| \rho \by \|_{H^2_x L^2_y} \|A \rho \partial_{\by} \bu\|_{L^2} \notag\\
&\leq C (1+ \tau(0)^{-4}) \|\bu\|_{X_\tau}^{3/2} \|\bu\|_{Z_\tau}^{1/2} + C (1+\tau(0)^{-2}) \|\bu \|_{X_\tau} \|\bu\|_{Z_\tau}
\label{eq:vort:a:N}
\end{align}
where we have also used  $\tau(0)/2 \leq \tau(t) \leq \tau(0)$. Similarly, recalling the definition of the linear term in \eqref{eq:PR:L} one may show that 
\begin{align}
\| \rho L(\bu) \|_{L^2}   
&\leq C_U \| \rho \bu\|_{H^1_x L^2_y} + C_U \|A \rho \partial_{\by} \bu\|_{L^2_x L^2_y} \notag\\
&\leq C_U (1+\tau(0)^{-1}) \| \bu \|_{X_\tau} + C_U \|\bu\|_{Z_\tau}
\label{eq:vort:a:L}
\end{align}
and 
\begin{align}
\| \rho F \|_{L^2}   \leq C_U
\label{eq:vort:a:F}
\end{align}
where $ C_U = C_U (G_U,\tau_U,\alpha,\theta)>0$ is a constant. Since by construction (cf.~\eqref{eq:analytic:final:estimate}) we have that 
\begin{align*}
\sup_{t\in[0,T_a]} \| \bu(t)\|_{X_{\tau(t)}(I_a)}^2 + \int_0^{T_a} \| \bu (t) \|_{Z_{\tau(t)}}^2    dt \leq 1 + 2 \|\bu_0\|_{X_{\tau(0)}}^2 
\end{align*}
by combining \eqref{eq:vorticity:analytic} with \eqref{eq:vort:a:t} and \eqref{eq:vort:a:N}--\eqref{eq:vort:a:F}, we obtain
\begin{align*}
\sup_{t\in[0,T_a]} \| \rho(\by) \partial_{\by} \bu \|_{L^2(I_a \times \RR_+)}^2 
+ \int_0^{T_a} \| \rho(\by) A \partial_{\by \by} \bu(t) \|_{L^2(I_a \times \RR_+)}^2 dt
\leq C
\end{align*}
where $C>0$ depends on $C_U$, $\tau(0)$, and $\|\bu_0\|_{X_{\tau(0)}}$.
By translating back $\partial_{\by}$ to $\dy$ and $\bu$ to $u$, and using that $1/2 \leq A \leq 2$ on $I_a \times [0,T_a]$, we conclude the proof of the lemma.
\end{proof}

\subsubsection{Construction of the monotone solution}
\label{sec:P:monotone}

In order to explore certain nonlinear cancellations present in the two dimensional equations, following~\cite{MasmoudiWong12a} we look at the equation obeyed by the vorticity $\omega = \dy u$ which reads
\begin{align}
&\dt \omega - \partial_{yy} \omega + u \dx \omega + v \dy \omega = 0\\
&u(x,y,t) = U(x,t) - \int_{y}^\infty \omega(x,z,t) dz\\
&v(x,y,t) = - \int_0^y \dx u(x,z,t) dz\\
&\dy \omega(x,y,t)|_{y=0} = \partial_x P(x,t) \label{eq:omega:y=0}
\end{align}
supplemented with the initial condition 
\begin{align}
\omega(x,y,0) = \omega_0(x,y) = \partial_y u_0(x,y).
\end{align}

Due to the assumptions \eqref{eq:monotone:IC:1}--\eqref{eq:monotone:IC:2} on the initial velocity $u_0$ and initial vorticity $\omega_0$ on the intervals $I_m^+ = [\tau_0/2,\infty)$ and $I_m^- = (-\infty,-\tau_0/2]$, we may extend 
the functions $u_0(x,y)$ and $\omega_0(x,y)$ to the half-space $\RR \times \RR_+$, such that the {\em positive extension} $u_0^+, \omega_0^+, U^+$ obeys
\begin{align}
(u_0^+, \omega_0^+) = (u_0,\omega_0) \mbox{ on } I_m^+ \times \RR_+ \quad \mbox{and} \quad U^+(x,t) = U(x,t) \mbox{ on } I_m^+ \times [0,T]
\end{align}
and we have
\begin{align}
u_0^+ - U^+ \in H^{s,\gamma-1}_{\rm uloc}(\RR) \quad \mbox{and} \quad \omega_0^+ \in H^{s,\gamma}_{\sigma,3\delta/2,{\rm uloc}}(\RR)
\end{align}
with at most doubled uniform local norm.
Similarly, a {\em negative extension} $u_0^-, \omega_0^-, U^-$ may be defined such that 
\begin{align}
(u_0^-, \omega_0^-) = (u_0,\omega_0) \mbox{ on } I_m^- \times \RR \quad \mbox{and} \quad U^-(x,t) = U(x,y) \mbox{ on } I_m^- \times [0,T].
\end{align}
It obeys
\begin{align}
u_0^- - U^- \in H^{s,\gamma-1}_{\rm uloc} (\RR) \quad \mbox{and} \quad -\omega_0^- \in H^{s,\gamma}_{\sigma,3\delta/2,{\rm uloc}}(\RR).
\end{align}
We note that the corresponding extensions of the underlying Euler flow $U$ were necessary in order to maintain the compatibility conditions for $u_0^\pm(x,y)$ as $y \to \infty$.
The above mentioned extension is possible because functions in a Sobolev space $H^s$ may be localized (as opposed to real-analytic functions). It is also clear that the extension may be chosen in such a way that  the monotonicity constant $2\delta$ of \eqref{eq:Linfty:uloc} shrinks by a given factor.

With the above extensions in mind, we now solve the Prandtl equations on the half-space $\RR\times \RR_+$ with the corresponding initial data $(u_0^+,\omega_0^+)$ and $(u_0^-,\omega_0^-)$. The existence on a short time interval $[0,T_m]$ of a ``positive'' solution $(u^+,\omega^+)$ and a ``negative'' solution $(u^-,\omega^-)$ from these initial conditions is then guaranteed directly by  Theorem~\ref{thm:Oleinik:global} above. The solutions obey 
\begin{align}
u^\pm - U^\pm \in L^\infty([0,T_m];H^{s,\gamma-1}_{\rm uloc}(I_m^\pm \times \RR_+)) 
\end{align}
and
\begin{align}
\omega^\pm \in L^\infty([0,T_m];H^{s,\gamma}_{\sigma,\delta, {\rm uloc}}(I_m^\pm\times \RR_+)).
\end{align}

\subsubsection{Glueing the analytic and monotone solutions}
\label{sec:P:glueing}

Let $T_* = \min\{ T_a,T_m\}$. In the above sections we have constructed a unique real-analytic solution $u$ on $[-\tau_0,\tau_0] \times \RR_+ \times [0,T_*)$,  which may be computed from $\bar u$ via \eqref{eq:bu:def} and two monotone solutions $u^+$ and $u^-$ on $\RR\times \RR_+ \times [0,T_*)$. We shall now glue these three solutions in a suitable way in order to obtain a single solution $u^P$ of the Prandtl system on $\RR \times \RR_+ \times[0,T_*)$ which agrees with the positive solution for $x \in I_m^+= [\tau_0/2,\infty)$, with the negative one for $x \in I_m^- = (-\infty,-\tau_0/2]$, and with the analytic one for $x \in I_a = [-\tau_0,\tau_0]$. In order to achieve this, it is clear that first we need to prove that the the analytic solution agrees with the monotone solutions on the domain on which they overlap. The difficulty of not being able to localize real-analytic functions is overcome by using the {finite speed of propagation} in the equations {with respect to the $x$-variable} in the vorticity equation. 

We shall only give details for the overlap of $u$ and $u^+$, the case of $u$ and $u^-$ being the same.
By definition we have that $u_0 = u_0^+$ on $[\tau_0/2,\tau_0]\times \RR^+$, and therefore the initial vorticities match as well, i.e., $\omega_0 = \omega_0^+$. We shall now prove that there exists $M>0$ such that $u(\cdot,t) = u^+(\cdot,t)$ and $\omega(\cdot,t) = \omega^+(\cdot,t)$ on the strip 
\[
I_t \times \RR_+= [\tau_0/2 + M t  ,\tau_0 - M t] \times \RR_+,
\]
for all $t \in [0,T^*)$, where $T^* = \min \{ T_*, \tau_0/(4M) \}$.

For this purpose we localize the uniqueness argument in~\cite[Section 6.2]{MasmoudiWong12a}, but only in the $x$ variable. We let 
\begin{align}
\tu = u - u^+ \quad \mbox{and} \quad \tom = \omega - \omega^+,
\end{align}
where $\omega$ is the solution which is real-analytic with respect to $x$ and has $W^{1,2}$ regularity in the $y$-variable by Lemma~\ref{lem:vort:dy}, and $\omega^+(x,y) \geq \delta (1+y)^{-\sigma} > 0$  is the positive $H^s$ solution. The equation obeyed by $\tilde \omega$ reads
\begin{align}
\dt \tom  - \dyy \tom + u \dx \tom + v \dy \tom + \tilde u \dx \omega^+ + \tilde v \dy \omega^+ = 0,
\label{eq:tilde:omega}
\end{align}
where 
\begin{align*}
u = \dy^{-1} \omega, \quad  v = - \dx \dy^{-1} u, \quad \tilde u = \dy^{-1} \tom, \quad \tilde v = - \dx \dy^{-1} \tilde u.
\end{align*}
while the equation obeyed by $\tu$ is
\begin{align}
\dt \tu - \dyy \tu + u \dx \tu + v \dy \tu + \tu \dx u^+ + \tilde v \dy u^+ = 0.
\label{eq:tilde:u}
\end{align}
In order to explore the nonlinear cancellation which permits solving the Prandtl equation in Sobolev spaces, we consider the function
\begin{align}
g = \tom - \tu \frac{\dy \omega^+}{\omega^+} = \tom - \tu \Omega^+= \omega^+ \dy \left( \frac{\tu}{\omega^+}\right)
\label{eq:g:def}
\end{align}
where we denoted 
\[
\Omega^+ = \dy \log(\omega^+)
\] 
which by construction obeys
\begin{align}
(1+y) |\Omega^+(x,y)| + (1+y) |\dx \Omega^+(x,y)| + (1+y)^2 |\dy \Omega^+(x,y)| \leq \delta^{-2} + \delta^{-4}
\label{eq:OMEGA:bounds}
\end{align}
for $(x,y) \in I_t \times \RR_+$.

Our goal is to show that 
\begin{align}
\| g\|_{L^2(I_t \times \RR_+)}^2 = 0 \label{eq:patching:to:do}
\end{align}
for all $t \in [0,T^*)$, for a suitably chosen constant $M$ in the definition of $I_t$.

First, we explain why $g(t) = 0$ on $I_t \times \RR_+$ implies $\tu = \tom = 0$ on this shrinking strip.
If $g = 0$, since $\omega^+ \geq \delta (1+y)^{-\sigma} > 0$, it follows by \eqref{eq:g:def} that
\begin{align*}
\tilde u(x,y,t) = \omega^+(x,y,t) f(x,t)
\end{align*}
for some function $f(x,t)$. Since by construction we have $\delta \leq  \omega^+ (x,0) \leq \delta^{-2}$ it follows that $f(x,t) = 0$, since in view of the existing boundary condition $\tilde u (x,0,t) = 0$. This proves that $\tilde u(t) = 0$ and thus $\tom(t) = 0$ on $I_t \times \RR_+$, as desired.

It is left to prove \eqref{eq:patching:to:do}. The equation obeyed by $g$ is
\begin{align}
\left( \dt - \dyy + u \dx + v \dy \right) g = - 2 g \dy \Omega^+ - \tu \left( \tu \dx \Omega^+ + \tilde v \dy \Omega^+ \right)
\label{eq:g:PDE}
\end{align}
(cf.~the equation (6.13) in \cite{MasmoudiWong12a} for details).
The boundary conditions are given by 
\begin{align}
\left( \dy g + \Omega^+ g\right)|_{y=0} = 0 \quad \mbox{and} \quad \lim_{y\to \infty} g = 0.
\label{eq:g:BC}
\end{align}
To obtain the boundary condition for $g$ at $y=0$ we use \eqref{eq:P:4} and \eqref{eq:omega:y=0}. For $y\to \infty$ we have used that $\Omega^+ \to 0$ and $\tilde u = 0$ as $y \to \infty$,  and that $\tilde \omega = \omega -\omega^+ \to 0 - 0 = 0$ as $y\to \infty$, since by continuity  the real-analytic solution is also monotone on $[0,T^*)$.

Upon multiplying \eqref{eq:g:PDE} with $g$ and integrating over $I_t \times \RR_+$ we obtain
\begin{align}
&\frac 12 \frac{d}{dt} \int_{I_t \times \RR_+}  g(x,y,t)^2 dx dy +\frac{M}{2}  \int_{\RR_+} \left( g(\tau_0/2 + Mt,y)^2 + g(\tau_0 - Mt,y)^2 \right) dy \notag \\
&\qquad =  \int_{I_t \times \RR_+} g \left(\partial_{yy} g -  u \dx g - v \dy g - 2 g \dy \Omega^+ - \tu^2  \dx \Omega^+ +\tu  \tilde v \dy \Omega^+  \right) dx dy
\notag \\
&= J_1 + J_2 + J_3 + J_4 + J_5 + J_6. 
\label{eq:g:ODE:1}
\end{align}
Integrating by parts in the $y$-variable and using \eqref{eq:g:BC} we arrive at 
\begin{align}
J_1 &= - \int_{I_t \times \RR_+} |\dy g|^2  dx dy  + \int_{I_t} \Omega^+(x,0) g(x,0)^2 dx \notag\\
&\leq - \frac 12 \| \dy g\|_{L^2(I_t \times \RR_+)}^2     +  C_\delta \| g\|_{L^2(I_t \times \RR_+)}^2
\label{eq:J1}
\end{align} 
by appealing to the trace theorem $\|g(\cdot,0)\|_{L^2(I_t)} \leq C \| g \|_{L^2(I_t \times \RR_+)} \| \dy g \|_{L^2(I_t \times \RR_+)}$ and estimate \eqref{eq:OMEGA:bounds}.

For the first pair of transport terms along the flow of the analytic solution, we integrate by parts with respect to $x$ and $y$ respectively, use the incompressibility condition $\dx u + \dy v = 0 $ and the boundary conditions on $u,v,g$ to obtain that
\begin{align}
J_2 + J_3
&=  \frac 12 \int_{\RR^+} u(\tau_0/2+Mt,y) g(\tau_0/2+Mt,y)^2 -  u(\tau_0-Mt,y)  g(\tau_0-Mt,y)^2 dy \notag\\
& \leq \frac 12 \left( \sup_{(x,y) \in I_a\times\RR_+} |u| \right) \int_{\RR_+} \left( g(\tau_0/2 + Mt,y)^2 + g(\tau_0 - Mt,y)^2 \right) dy \notag\\
&\leq \frac M2   \int_{\RR_+} \left( g(\tau_0/2 + Mt,y)^2 + g(\tau_0 - Mt,y)^2 \right) dy
\label{eq:J23}
\end{align}
by choosing $M$ large enough. In particular, using Lemma~\ref{lem:vort:dy} we obtain the bound
\begin{align}
\sup_{(x,y) \in I_a\times\RR_+} |u| \leq \| \rho \omega\|_{L^2_x L^2_y}^{1/2} \|\rho \dx \omega \|_{L^2_x L^2_y}^{1/2} \leq C_{\omega,a}
\label{eq:analytic:u:L:infty}
\end{align}
which is possible since $\rho(y) = \langle y \rangle^\alpha$ and $\alpha>1/2$, and therefore it is sufficient to choose 
\begin{align}
M \geq C_{\omega,a}
\label{eq:M:choice}
\end{align}
in order to absorb $J_2+J_3$ into the right hand side of \eqref{eq:g:ODE:1}. For the term $J_4$, directly from H\"older and \eqref{eq:OMEGA:bounds} we have
\begin{align}
J_4 \leq C_\delta \|g\|_{L^2(I_t\times \RR_+)}^2 
\label{eq:J4}
\end{align}
For $J_5$, using \eqref{eq:Linfty:uloc}, \eqref{eq:OMEGA:bounds}, and \eqref{eq:analytic:u:L:infty} we have
\begin{align}
J_5 
&\leq \|g\|_{L^2(I_t \times \RR_+)} \|(1+y)^{-1} \tilde u\|_{L^2(I_t \times \RR_+)} \| (1+y) \dx \Omega^+ \|_{L^\infty(I_t \times \RR_+)} \| u - u^+ \|_{L^\infty(I_t \times \RR_+)} \notag \\
&\leq C \|g\|_{L^2} \|(1+y)^{-1} \tilde u\|_{L^2}
\end{align}
with $C$ depending on $\delta$, the norm of $\omega^+$ in $L^\infty([0,T^*];H^{s,\gamma}_{\sigma,2\delta}(I_m^+ \times \RR_+))$ and on $C_{\omega,a}$. Now, similarly to \cite[Claim 6.5]{MasmoudiWong12b}, using that $\delta \leq (1+y)^\sigma \omega^+(x,y) \leq \delta^{-1}$,  the boundary condition on $\tilde u$, and integrating by parts in $y$, we estimate
\begin{align}
\|(1+y)^{-1} \tilde u\|_{L^2(I_t \times \RR_+)} 
&= \left\| (1+y)^\sigma \omega^+ (1+y)^{-1-\sigma} \frac{\tilde u}{\omega^+} \right\|_{L^2(I_t \times \RR_+)}   \notag\\
&\leq C_{\delta,\sigma} \left\| (1+y)^{-\sigma} \dy \left( \frac{\tilde u}{\omega^+} \right) \right\|_{L^2(I_t \times \RR_+)} \notag\\
&\leq C_{\delta,\sigma} \|g\|_{L^2(I_t \times \RR_+)}.
\label{eq:u:to:g:trick}
\end{align}
Thus we obtain
\begin{align}
J_5 \leq C \|g\|_{L^2(I_t \times \RR_+)}^2.
\label{eq:J5}
\end{align}
Similarly, we have
\begin{align}
J_6 
&\leq \|g\|_{L^2(I_t \times \RR_+)} \|(1+y)^{-1} \tilde u\|_{L^2(I_t \times \RR_+)} \| (1+y)^2 \dy \Omega^+ \|_{L^\infty(I_t \times \RR_+)} \| (1+y)^{-1} (v - v^+) \|_{L^\infty(I_t \times \RR_+)} \notag\\
&\leq C \|g\|_{L^2(I_t \times \RR_+)}^2
\label{eq:J6}
\end{align}
by using \eqref{eq:u:to:g:trick} and the bounds available on $(1+y)^{-1} v$ and $(1+y)^{-1} v^+$ in $L^\infty([0,T^*];L^\infty(I_t \times \RR_+))$.

Combining \eqref{eq:g:ODE:1}, \eqref{eq:J1}, \eqref{eq:J23}, \eqref{eq:J4}, \eqref{eq:J5}, and \eqref{eq:J6}, we obtain that if $M$ is chosen so that it exceeds the maximal velocity of the analytic solution, i.e.,~\eqref{eq:M:choice} holds, then by Gr\"onwall
\begin{align}
\|g(t)\|_{L^2(I_t\times \RR_+)} \leq \|g(0)\|_{L^2(I_0 \times \RR_+)} \exp(C t) 
\label{eq:g:Gronwall}
\end{align}
which concludes the proof of \eqref{eq:patching:to:do}, since $g_0 = 0$. This concludes the proof that the real-analytic and the monotone solutions agree for all $t \in [0,T^*]$ for $(x,y) \in I_{T^*}\times \RR_+ =  [\tau_0/2 + M T^*, \tau_0 - M T^*]$. Therefore we can patch the analytic and monotone solutions together, and we obtain a global in $x$, local in time, solution of the Prandtl system.

\subsubsection{Uniqueness}
\label{sec:P:uniqueness}
The uniqueness holds in the sense that two solutions in the class given by the theorem are the same. To see this, note that the interval of monotonicity $I_m^+$ (respectively $I_m^-$) overlaps with the interval of analyticity $I_a$, with an overlap of initial size $\tau_0/2$. Using an argument that is identical to the proof of the glueing one given in subsection~\ref{sec:P:glueing} above, we may thus establish the uniqueness of solutions on shrinking monotonicity and analyticity intervals. However, since the speed of this shrinking is finite, by letting the time of existence be sufficiently small, the proof of uniqueness is established.

\section{The main result for the hydrostatic Euler equations} \label{sec:HE:proof}

In the two dimensional setting it is convenient to study the evolution of the vorticity $\omega = \dy u$. Indeed, applying $\dy$ to \eqref{eq:HE:1}, and using \eqref{eq:HE:2}--\eqref{eq:HE:3} we obtain the nonlinear transport equation
\begin{align} 
\dt \omega + u \dx \omega + v \dy \omega = 0 \label{eq:vorticity},
\end{align}
where, using the notation in \cite{MasmoudiWong12b}, one may compute $(u,v)$ from $\omega$ via
\begin{align} 
u = - \dy \AA(\omega) \mbox{ and } v = \dx \AA(\omega) \label{eq:Biot-Savart}
\end{align}
where the stream function $\AA(\omega)$ solves 
\begin{align*}
- \dyy \AA(\omega) = \omega
\end{align*} 
with the boundary condition 
\begin{align*}
\AA(\omega)|_{y=0,1} = 0.
\end{align*} 
Since we are working in the setting where \eqref{eq:mean} holds, it is not difficult to verify that in the smooth category $(u,v,p)$ solves \eqref{eq:HE:1}--\eqref{eq:HE:4} and \eqref{eq:mean}, if and only if $(u,v,\omega)$ solves \eqref{eq:vorticity}--\eqref{eq:Biot-Savart}. We also have the following estimates for $u$ and $v$ in terms of $\omega$. 

\begin{lemma}  
\label{lemma:u:v:omega}
Let $\alpha = (\alpha_1, \alpha_2) \in {\mathbb N}_0^2$ be a multi-index, and let $u$ and $v$ be determined from the smooth function $\omega$ via \eqref{eq:Biot-Savart}. Also let $\DD '$ be a cylindrical subset of $\DD$, i.e., $\DD' = \Omega \times (0,1)$ for some open set $\Omega \subset \RR$.  If $\alpha_2 =0$, we have
\begin{align}
\|\partial^\alpha u \|_{L^p(\DD')}  \leq C \| \dx^{\alpha_1} \omega \|_{L^p(\DD')} \leq C \| \omega\|_{W^{|\alpha|,p}(\DD')} \label{eq:u:w:Bad}
\end{align}
while if $\alpha_2 \geq 1$ we may bound
\begin{align}
\|\partial^\alpha u \|_{L^p(\DD')}  =  \| \dx^{\alpha_1} \dy^{\alpha_2 -1} \omega \|_{L^p(\DD')} \leq C \| \omega\|_{W^{|\alpha|-1,p}(\DD')}\label{eq:u:w}
\end{align}
for all $2 \leq p \leq \infty$. 
Similarly, if $\alpha_2 = 1$, then 
\begin{align}
\| \partial^\alpha v \|_{L^p(\DD')} \leq C  \| \dx^{|\alpha|} \omega \|_{L^p(\DD')} \leq C \| \omega\|_{W^{|\alpha|,p}(\DD')} \label{eq:v:w:BAD}
\end{align}
and if $\alpha_2 \geq 2$, we may estimate
\begin{align}
\| \partial^\alpha v \|_{L^p(\DD')}=  \| \dx^{\alpha_1 +1} \dy^{\alpha_2 -2} \omega \|_{L^p(\DD')} \leq C \| \omega\|_{W^{|\alpha|-1,p}(\DD')} \label{eq:v:w}
\end{align}
for all $2 \leq p \leq \infty$. 
\end{lemma}
\begin{proof}[Proof of Lemma~\ref{lemma:u:v:omega}]
Let us first prove the estimates on $u$. If $\alpha_2 \geq 1$, using $\omega = \dy u$ we have 
$$\partial^\alpha u = \dx^{\alpha_1} \dy^{\alpha_2} u = \dx^{\alpha_1} \dy^{\alpha_2 -1} \omega,$$ which proves \eqref{eq:u:w}. If $\alpha_2 = 0$, we use $\int_0^1 \dx^{\alpha_1} u(x,y) dy = 0$. Hence, using the Poincar\'e inequality in the $y$ variable, we have $\| \dx^{\alpha_1} u\|_{L^p} \leq C \| \dx^{\alpha_1} \dy u\|_{L^p}$, which proves \eqref{eq:u:w:Bad}.

In order to estimate $v$, note that when $\alpha_2 \geq 2$ we have 
$$\partial^\alpha v = \dx^{\alpha_1} \dy^{\alpha_2} v = \dx^{\alpha_1+1} \dy^{\alpha_2} \AA(\omega) =  - \dx^{\alpha_1+1} \dy^{\alpha_2-2} \omega,$$ which proves \eqref{eq:v:w}. On the other hand, if $\alpha_2=1$, we have $\dx^{\alpha_1} \dy^{\alpha_2} v = - \dx^{\alpha_1+1} u$, which can be estimated in $L^p$ via \eqref{eq:u:w:Bad}, thereby proving \eqref{eq:v:w:BAD}.
\end{proof}

For simplicity of the presentation we assume that for $(x,y)$ belonging to the {\em monotonicity} domain
\begin{align*}
\DDM = (-\infty, 1) \times (0,1) 
\end{align*} 
we have that 
  \begin{equation}
   0 < \sigma \leq  \dy \omega_0(x,y) \leq \sigma^{-1}
   \label{EQ01}
\end{equation}
for some $\sigma \in (0,1)$, while for $(x,y)$ in the 
{\em analyticity} domain
\begin{align*}
\DDA =  (0,\infty) \times (0,1) 
\end{align*} 
we have that $u_0$ is uniformly real analytic
with radius $\tau_0>0$, and analytic norm bounded by a constant $M>0$.

Following the notation in \cite{MasmoudiWong12b}, for $\sigma>0$ and $s\geq 4$ we consider the Rayleigh-modified Sobolev space
\begin{align*} 
H_{\sigma}^s (I) = \{ \omega \in H^s(I \times (0,1)) : \sigma \leq \dy \omega \leq \sigma^{-1} \}
\end{align*}
with the norm
\begin{align} 
\| \omega \|_{H^s_{\sigma}}^2 = \left\| \frac{ \dx^s \omega }{\sqrt{\dy \omega} } \right\|_{L^2}^2 + \sum_{|\alpha| \leq s, \alpha_1 \neq s} \| \partial^\alpha \omega \|_{L^2}^2 \label{eq:Hs:norm}.
\end{align}
Here $\alpha = (\alpha_1,\alpha_2)$, with $\alpha_1, \alpha_2 \geq 0$, is a two dimensional multi-index, and $\partial^\alpha = \dx^{\alpha_1} \dy^{\alpha_2}$. For functions in $H_\sigma^s$, it is clear that the norm $\| \cdot \|_{H^s_\sigma}$ and the usual Sobolev norm $\| \cdot \|_{H^s}$ are equivalent, with an equivalence constant that depends on $\sigma$.
Now, write
$I=(-\infty,1)$
as
$I = \bigcup_{j \in{\mathbb N}} I_j$ where
$I_j=(-j-1,-j+1)$.
Then define
\begin{align}
H^{s}_{\sigma,\rm uloc}(I) =
\Bigl\{ \omega \colon I \times (0,1) \to \RR,
    \sigma \leq \dy \omega \leq \sigma^{-1},
    \|\omega\|_{H^{s}_{\sigma,\rm uloc}(I)} = \sup_{j\in{\mathbb N}} \|\omega\|_{H^{s}_{\sigma}(I_j)} < \infty
   \Bigr\}.
\label{EQ08}
\end{align}
For the analytic part of the solution we use the notation from \cite{KukavicaTemamVicolZiane11}. For $\tau>0$, we  define the space of real analytic functions with the analyticity radius $\tau$ as 
\begin{align*} 
X_\tau (\DDA) = \left \{ \omega \in C^\infty(\DDA) : \| \omega \|_{X_\tau} < \infty \right\}
\end{align*}
where the analytic norm is defined as 
\begin{align} 
\| \omega \|_{X_\tau}^2 = \sum_{m\geq 0} \frac{\tau^{2m} (m+1)^2}{m!^2} \| \omega \|_{\dot{H}^m(\DDA)}^2.
\end{align}
Above and throughout the paper, $\dot{H}^m$  denotes the homogeneous Sobolev space with the (semi) norm 
\begin{align*}
\|\omega\|_{\dot{H}^m(\DDA)}^2 = \sum_{|\alpha|=m} \| \partial^\alpha \omega \|_{L^2(\DDA)}^2.
\end{align*}
It is also convenient to introduce the space
\begin{align*} 
Y_\tau (\DDA) = \{ \omega \in X_\tau : \| \omega \|_{Y_\tau} < \infty \},
\end{align*}
where
\begin{align} 
\| \omega \|_{Y_\tau}^2 = \sum_{m\geq 1} \frac{m \tau^{2m-1} (m+1)^2}{m!^2}\| \omega \|_{\dot{H}^m(\DDA)}^2 .
\end{align}

The first main theorem may then be stated as follows.

\begin{theorem}\label{thm:hydro:main}
Let $\sigma,\tau_0>0$ and $s\geq 4$. Assume that the initial vorticity satisfies a uniform Rayleigh condition on $\DDM$, we have $\omega_0 \in H^s_{2\sigma, \rm uloc}(\DDM)$, and $\omega_0$ is uniformly real analytic on $\DDA$ with radius of analyticity at least $\tau_0$, i.e., $u_0 \in Y_{\tau_0}(\DDA)$. Also assume that the initial velocity $u_0$ satisfies the compatibility condition  \eqref{eq:mean}. Then there exist  $T>0$ and a unique smooth solution $\omega \in C(0,T; X_{\tau_0/2}(\DDA) \cap  H^s_{\sigma,\rm uloc}(\DDM))$  of \eqref{eq:vorticity}--\eqref{eq:Biot-Savart} on $[0,T]$, which has zero vertical mean.
\end{theorem}

\begin{remark} \label{thm:hydro:general} The above theorem can be generalized to allow
for regions where the initial vorticity is either increasing or decreasing. Let $\sigma,\tau_0>0$  and $s\geq 4$. Assume there exists open intervals $\{L_{i} \}_{i=1}^{n} \subset \RR$ such that on $\DD_{m,i} =L_{i} \times (0,1)$ the initial vorticity $\omega_0(x,y)$ is strictly monotone with respect to $y$, and either $\omega_0 \in H^s_{2\sigma,{\rm uloc}}(\DD_{m,i})$ or $- \omega_0 \in H^s_{2\sigma,{\rm uloc}}(\DD_{m,i})$ for all $i\in \{1,\ldots,n\}$. Assume also that there exist open intervals 
$\{J_j\}_{j=1}^{m} \subset \RR$ such that $(\cup_{j} J_j) \cup (\cup_i L_i) = \RR$ and such that on $\DD_{a,j} = J_j \times (0,1)$ the initial vorticity $\omega_0$ is uniformly real analytic with radius at least $\tau_0$, that is, $\omega_0 \in Y_{\tau_0}(\DD_{a,j})$ for all $j \in \{1,\ldots,m\}$. Assuming the initial velocity $u_0$ obeys the compatibility condition \eqref{eq:mean}, there exist $T>0$ and a smooth solution $\omega(t)$ of \eqref{eq:vorticity}--\eqref{eq:Biot-Savart} on $[0,T]$ which is monotone on the $\DD_{m,i}$  and is real analytic on $\DD_{a,j}$ .
We note that it is possible to accommodate for an infinite number of intervals , i.e., $m$ and $n$ may be $\infty$.
\end{remark}

Below we present the proof of Theorem~\ref{thm:hydro:main}, the proof of
the general case described in Remark~\ref{thm:hydro:general}, being completely analogous.
The argument closely follows the ideas in the Prandtl section, with the main work having to be done when glueing the real-analytic and the convex solutions.
To avoid redundancy, we omit the other details of the proof.

\begin{proof}[Proof of Theorem~\ref{thm:hydro:main}]
Let $\bar\omega_0$ denote a function which agrees with $\omega_0$ on
the monotonicity domain
${\mathcal D}_m=(-\infty,-1)\times(0,1)$ but which satisfies the
convexity condition \eqref{EQ01} on the whole strip ${\mathbb R}\times(0,1)$.
Denote by $\omega$ the solution of the initial value problem \eqref{eq:vorticity}--\eqref{eq:Biot-Savart} with initial data $\bar \omega_0$ on $\RR\times (0,1)$, obtained from Theorem~2.5 in 
\cite{MasmoudiWong12b} and the arguments used to prove Theorem~\ref{thm:Oleinik:global}, on some time interval $(0,t_0)$. Note that in order to construct the solution $\omega$, as in Section~\ref{sec:P:Oleinik} above, one needs to first use \cite{MasmoudiWong12b} to construct countably many solutions $\omega^{(j)}$ which are $(j-1,j+3)$ periodic, where $j \in \ZZ$; due to the uniform local Rayleigh condition on the initial data, all these solutions may be shown to live on a common time-interval, with a uniform in $j$ bound on their $H^{s}_\sigma$ norms; the next step is then to glue these solutions together, which is achieved using precisely the argument presented below in this section. To avoid redundancy with the proof of Theorem~\ref{thm:Oleinik:global}, we omit further details for the construction of the monotone solution $\omega$.

On the other hand, by Theorem~2.2 \cite{KukavicaTemamVicolZiane11}, there exists an analytic solution
$\tilde\omega$ on the spatial domain
${\mathcal D}_a$
also on the same time interval
$(0,t_0)$, without loss of generality.

Denote $U=\tilde u-u$,
$V=\tilde v-v$, and
$\Omega=\tilde \omega-\omega$. Then
  \begin{equation}
   \partial_{t}\Omega
   +\tilde u\partial_{x}\Omega
   +\tilde v\partial_{y}\Omega
   +U\partial_{x}\omega
   +V\partial_{y}\omega
   =0
   .
   \label{EQ02}
  \end{equation}
With $M>0$ to be determined, write
  \begin{equation*}
   I_t
   =(M t, 1-M t)
   .
  \end{equation*}
We now show that if $M$ is a sufficiently large constant, and $t_0$ is
sufficiently small, the solutions $\omega$ and
$\tilde\omega$ agree on 
$I_t \times(0,1)$ for all $t \in [0,t_0)$.
For this purpose, evaluate a
weighted
norm
of $\Omega(x,y,t)^2$ for $(x,y)\in I_t\times(0,1)$.
Differentiating the quantity
  \begin{equation*}
   X(t)
   = \frac12
    \iint_{I_t\times(0,1)}
      \frac{\Omega^2}{\partial_{y}\omega}
     dx dy
   = \frac12
    \int_{0}^{1}dy
     \int_{M t}^{1-M t}
      \frac{\Omega^2}{\partial_{y}\omega}
     dx,
  \end{equation*}
we obtain
  \begin{align}
   X'(t)
   &= 
   -\frac{M}{2}
   \int_{0}^{1}
    \frac{
     \Omega(1-M t,y,t)^2
        }{
     \partial_{y}\omega(1-M t,y,t)
    }
   dy
   -\frac{M}{2}
   \int_{0}^{1}
    \frac{
     \Omega(M t,y,t)^2
        }{
     \partial_{y}\omega(M t,y,t)
    }
   dy
   \nonumber\\&\indeq
   +
    \iint_{I_t\times(0,1)}
    \frac{
     \Omega\partial_{t}\Omega
        }{
     \partial_{y}\omega
    }
    dx dy
    -
    \iint_{I_t\times(0,1)}
    \frac{
     \Omega^2 \partial_t \partial_y \omega
        }{
     (\partial_{y}\omega)^2
    }
    dx dy
   \label{EQ04a}
  \end{align}
and thus
  \begin{align}
   X'(t)
   &= 
   -\frac{M}{2}
   \int_{0}^{1}
    \frac{
     \Omega(1-M t,y,t)^2
        }{
     \partial_{y}\omega(1-M t,y,t)
    }
   dy
   -\frac{M}{2}
   \int_{0}^{1}
    \frac{
     \Omega(M t,y,t)^2
        }{
     \partial_{y}\omega(M t,y,t)
    }
   dy
   \nonumber\\&\indeq
   -
    \iint_{I_t\times(0,1)}
    \frac{
     \tilde u \Omega \partial_{x} \Omega
        }{
     \partial_{y}\omega
    }
    dx dy
    -
    \iint_{I_t\times(0,1)}
    \frac{
     \tilde v \Omega\partial_{y}\Omega
        }{
     \partial_{y}\omega
    }
    dx dy
   \nonumber\\&\indeq
   -
    \iint_{I_t\times(0,1)}
    \frac{
     U\Omega\partial_{x}\omega
        }{
     \partial_{y}\omega
    }
    dx dy
   -
    \iint_{I_t\times(0,1)}
    V \Omega
    dx dy
   \nonumber\\&\indeq
    -
    \iint_{I_t\times(0,1)}
    \frac{
      \Omega^2 \partial_t \partial_{y}\omega
        }{
     (\partial_{y}\omega)^2
    }
    dx dy
    \nonumber\\&
    =
    I_1+I_2+I_3+I_4+I_5+I_6+I_7
   .
   \label{EQ04}
  \end{align}
For $I_1$ and $I_2$, we have
  \begin{equation*}
   I_1+I_2
   \le
   -\frac{M}{C}
   \int_{0}^{1}
    \Omega(1-M t,y,t)^2
   dy
   -\frac{M}{C}
   \int_{0}^{1}
    \Omega(M t,y,t)^2
   dy
  \end{equation*}
where $C$ denotes a sufficiently large generic constant
which may depend on $\sigma$ and the initial data.
In the term $I_3$, we integrate by parts in the $x$ variable
and obtain
  \begin{align}
   I_3
   &=
   \frac12
   \iint_{I_t\times(0,1)}
   \Omega^2
   \partial_{x}\left(
                \frac{
                 \tilde u
                    }{
                 \partial_{y}\omega
                }
               \right)
   dx dy
   -\frac12
   \int_{0}^{1}
   \Omega(1-M t,y,t)^2
   \frac{
    \tilde u(1-M t,y,t)
       }{
    \partial_{y}\omega(1-M t,y,t)
   }
   dy
   \nonumber\\&\indeq
   +\frac12
   \int_{0}^{1}
   \Omega(M t,y,t)^2
   \frac{
    \tilde u(M t,y,t)
       }{
    \partial_{y}\omega(M t,y,t)
   }
   dy
   \nonumber\\&
   \le
   C
   \Vert \Omega\Vert_{L^2(I_t\times(0,1))}^2
   +
   C\int_{0}^{1}
     \Omega(1-M t,y,t)^2
    dy
   +
   C\int_{0}^{1}
     \Omega(M t,y,t)^2
    dy
   .
   \label{EQ03}
  \end{align}
For $I_4$, we integrate by parts in the $y$ variable. 
No boundary terms appear due to vanishing of $\tilde v$
on $y=0$ and $y=1$, and we get
  \begin{align}
   I_4
    &=
    \frac12
     \iint_{I_t\times(0,1)}    
     \Omega^2 
     \partial_{y}
       \left(
         \frac{
          \tilde v
             }{
          \partial_{y}\omega
         }    
       \right)
     dx dy
   \le 
   C   \Vert \Omega\Vert_{L^2(I_t\times(0,1))}^2
   \label{EQ05}
  \end{align}
For $I_5$, we use 
$\Vert U\Vert_{L^2(I_t\times(0,1))}
\le C \Vert \Omega\Vert_{L^2(I_t\times(0,1))}$, and we obtain
  \begin{equation*}
   I_5
   \le
   C\Vert \Omega\Vert_{L^2}^2
  \end{equation*}
For $I_6$, we integrate by parts in $y$ and then in the $x$
variable. Again, since $v|_{y=0,1}=0$, the boundary terms only appear when integrating by
parts in the $x$ variable, and we get
  \begin{align}
   I_6
   &=
   -
   \iint_{I_t\times(0,1)}
    V\Omega
   dx dy
   = - \iint_{I_t\times(0,1)}
        V U_y
       dx dy
   =  \iint_{I_t\times(0,1)}
        V_y U
       dx dy
   \nonumber\\&
   =  -\iint_{I_t\times(0,1)}
       U_x U
       dx dy
   = -\frac12
     \int_{0}^{1}
     U(1-M t,y,t)^2
     dy
     +\frac12
     \int_{0}^{1}
     U(M t,y,t)^2
     dy
   .
   \label{EQ06}
  \end{align}
Finally, by using the equation obeyed by $\partial_t \partial_y \omega$, for $I_7$ we have
  \begin{align*}
   I_7
   \le
   C    \Vert \Omega\Vert_{L^2(I_t\times(0,1))}^2
  \end{align*}
  for a sufficiently large constant $C$.
Collecting all the estimates on all the terms in
\eqref{EQ04} then leads to
  \begin{equation*}
   X'(t)
   \le
   \left(
    C-\frac{M}{C}
   \right)
   \int_{0}^{1}
   \Omega(1-M t,y,t)^2
   dy
   +
   \left(
    C-\frac{M}{C}
   \right)
   \int_{0}^{1}
   \Omega(M t,y,t)^2
   dy
   + C X(t)
   .
  \end{equation*}
Now, choose $M$ so large that the first two terms on the right side
are negative. Then we get
$X'(t)\le C X(t)$
and thus $X(t)\equiv0$ for $t\le 1/C_0$,
where $C_0$ is a sufficiently large constant.
By reducing $t_0$ if necessary, we may assume that $t_0= 1/C_0$.
Now, define the solution
$\bar\omega$ on $(-\infty,\infty)\times(0,1)$ as follows:
Let $\bar\omega=\omega$ 
if $x<1-M t$ and $\bar\omega=\tilde \omega$ if $x\ge M t$.
For $M t<x<1-M t$, the solutions agree by the first part of the
proof. 
Therefore,
$\bar \omega$ provides a solution on the interval $(0,t_0)$.
Uniqueness follows in a similar fashion. 
\end{proof}

\subsection*{Acknowledgments} 
The work of IK was supported in part by
the NSF grant DMS-1311943, 
the work of NM was supported in part by the
NSF grant DMS-1211806,
while the work of VV was supported in part by the NSF grant DMS-1211828.

\end{document}